\documentclass{amsart}

\usepackage{amssymb}
\usepackage[dvips]{epsfig}
\usepackage{graphicx}
\usepackage{color}

\newcommand{\R}{{\mathbb R}}

\newcommand{\bv}{{\mathbf v}}

\newcommand{\curl}{\mbox{\textnormal{curl }}}

\theoremstyle{remark}

\newtheorem{remark}{Remark}

\theoremstyle{theorem}
\newtheorem{proposition}{Proposition}
\newtheorem{lemma}{Lemma}
\newtheorem{theorem}{Theorem}

\newtheorem*{merci}{Acknowledgements}
\DeclareMathOperator{\divg}{div}
\begin{document}
\title[Two-dimensional  surface waves Boussinesq systems]{Well-posednesss of strongly dispersive two-dimensional  surface waves Boussinesq systems}
\author{Felipe Linares}
\address{ IMPA\\ Estrada Dona Castorina 110\\ Rio de Janeiro 22460-320, RJ Brasil}
\email{ linares@impa.br}

\author{Didier Pilod}
\address{Instituto de Matem\' atica, UFRJ, Caixa Postal 68530 CEP 21941-97, Rio de Janeiro, RJ Brasil}
\email{didier@im.ufrj.br, pilod@impa.fr}
\author{Jean-Claude Saut}
\address{Laboratoire de Math\' ematiques, UMR 8628,\\
Universit\' e Paris-Sud et CNRS,\\ 91405 Orsay, France}
\email{jean-claude.saut@math.u-psud.fr}
\date{April 11, 2011}
\maketitle
\begin{abstract}
We consider in this paper the well-posedness for the Cauchy problem associated to two-dimensional dispersive systems
of Boussinesq type which model weakly nonlinear long wave surface waves. We emphasize the case of the {\it strongly dispersive}
ones with focus on the \lq\lq KdV-KdV" system which possesses the strongest dispersive properties and which is a vector two-dimensional
extension of the classical KdV equation.
\end{abstract}

\section{Introduction}

\subsection{General Setting}

In general, nonlinear dispersive equations and systems are not derived from first principles. They are obtained as asymptotic
models (normal forms) when some {\it small} parameter tends to zero, of more general systems, under a suitable scaling. There
are supposed to describe the dynamics under suitable scaling conditions. We will study in this article {\it small amplitude,
long wave models} which have the general form

\begin{equation}\label{general}
U_t+AU+\epsilon F(U, \nabla U)+\epsilon LU=0.
\end{equation}

Here $\epsilon$ is a \lq\lq small\rq\rq \ parameter which takes into
account the nonlinear and dispersive effects, which are assumed of
the same order. $U=U(x,t)$ is a vector in $\R^{n+1}$, $n=1,2$ and
$x\in \R^n,$ $t>0.$ The order zero part $U_t+AU$ is just the linear
wave equation, while the third and fourth terms in \eqref{general}
represents the nonlinear and dispersive effects. The right-hand side
of \eqref {general} should in fact be $O(\epsilon ^2)$.

A typical example is that of the {\it abcd} Boussinesq systems for long, small amplitude gravity surface water waves
introduced in \cite {BCS1}, \cite{BCL}

 \begin{equation}
    \label{abcd}
    \left\lbrace
    \begin{array}{l}
    \eta_t+\nabla \cdot {\bf v}+\epsilon \lbrack\nabla\cdot(\eta {\bf v})+a \nabla\cdot \Delta{\bf v}-b\Delta \eta_t\rbrack=0 \\
    {\bf v}_t+\nabla \eta+\epsilon\lbrack \frac{1}{2}\nabla |{\bf v}|^2+c\nabla \Delta \eta-d\Delta {\bf v}_t\rbrack=0.

    \end{array}\right.
    \end{equation}

    Here, $\eta$ is the deviation  of the free surface from its rest state and ${\bf v}$ is an $O(\epsilon^2)$ approximation of the horizontal velocity taken at a
    certain depth (see \cite{BCS1}, \cite{BCL}). The constants $a, b, c, d$ are modeling parameters subject to the constraint $a+b+c+d=\frac{1}{3}.$ Those three
    degrees of freedom arise from the choice of the height at which the velocity is taken and from a double use of the {\it BBM} trick.

    The small parameter $\epsilon$ is defined by

    $$\epsilon=a/h\sim (h/\lambda)^2,$$

    where $h$ denotes the mean depth of the fluid, $a$ a typical amplitude of the wave and $\lambda$ a typical horizontal wavelength.

    Those systems are  approximations of the full water wave system in the so-called {\it Boussinesq regime} (see \cite{BCS1}, \cite{BCL}).
    They degenerate into the KdV (or BBM) equation in the one-dimensional case, for waves traveling in one direction. They also appear as
    models for {\it internal waves}, in an appropriate regime (see \cite{BLS}, \cite{CGK}). When surface tension effects are taken into
    account, the coefficient $c$ should be changed into $c-\tau$ where $\tau\geq 0$ is the surface tension parameter (see \cite{DD}).
    Throughout the paper we will consider only the case of purely gravity waves, that is $\tau =0.$

    Of course restrictions are to be imposed on  $a, b, c, d$ in order that the {\it linear} part of \eqref{abcd} be well-posed. It was
    established in \cite{BCS2} that, when $n=1$, all the linearly well-posed systems are {\it locally} nonlinearly well-posed. As for the
    two-dimensional case ($n=2$),  it has been proved in  \cite{DMS1} that in the  {\it generic case} where $b>0,\; d>0$ one has
    well-posedness on time scales of order $O(\frac{1}{\sqrt{\epsilon}})$ for data in the Sobolev space  $H^1$. The local well-posedness
    in other cases (but not in the {\it strongly dispersive} \lq\lq KdV-KdV" case) was proved in \cite{CTA}, but the question of the dependence
    of the time existence with respect to $\epsilon$ was not addressed there.

     One of the goals of the present paper is to complete the local well-posedness theory for $a, b, c, d$ systems in two dimensions,
     with focus on the size of the lifespan of the solutions with respect to $\epsilon$.

     An interesting fact (already noticed in \cite{BCS1}) from the PDE view point is that, though the   $a, b, c, d$ systems
     describe the same wave propagation phenomenon, their {\it dispersive} properties are quite different. We will precise the
     dispersion matrix in Section 2. After diagonalization, one is led to two-dimensional systems with {\it strong dispersion}
     (of KdV or Schr\"{o}dinger type) or with {\it weak dispersion} (for instance of BBM type). One also obtains a system which
     can be viewed as a dispersive perturbation of the two-dimensional Saint-Venant (shallow-water) system, generalizing the one-dimensional
     system studied by Amick \cite{A} and Schonbek \cite{Sc}.

On the other hand, the solutions of the Cauchy problem associated to \eqref{abcd} should exist on time scales of order
$O(\frac{1}{\epsilon})$ in order to  prove that the asymptotic system is a good approximation of the water wave system \cite{AL}.
In fact, the solutions of \eqref{abcd} cease to be relevant as approximations of those of the original system on time scales larger
than $O(\frac{1}{\epsilon^2}).$

It is worth noticing that it is unlikely that all the {\it abcd} systems would have {\it global solutions}. In fact they are shown
to be Hamiltonian (and thus possess a formally conserved energy) only when $b=d$ (see \cite{BCS1}). This can be used in the one-dimensional
case to obtain the global well-posedness of the Cauchy problem for a few of the systems  see \cite{BCS2} but this situation is exceptional
and moreover does  not apply to the two-dimensional case where the Hamiltonian does not control any Sobolev type norm.

This situation is in strong contrast with {\it one-directional} models such as the Korteweg- de Vries (KdV) or Benjamin-Bona-Mahony (BBM)
equations or with {\it quasi one-directional} models such as the Kadomtsev-Petviashvili (KP) equation where global a priori bounds are
available (possibly at low degree of regularity), which allows to prove the global well-posedness of the Cauchy problem.

As far as we know the problem of solving {\it any} of the systems \eqref{abcd} on time intervals of order $O(\frac{1}{\epsilon})$
has been completely open until very recently, except for some one-dimensional ones which have global solutions (under some restrictions
on the initial data).

In \cite{BCL} Bona , Colin and Lannes obtained  an order $O(\frac{1}{\epsilon})$ existence interval for a {\it fully symmetric}
class of systems, in one and two dimensions. Those systems have a skew-adjoint linear part ($a=c$) and are obtained from \eqref{abcd}
via the nonlinear change of variables $\tilde{{\bf v}}={\bf v}(1+\frac{\epsilon}{2}\eta)$  which does not affect the linear part
(modulo higher order terms in $\epsilon$) and which symmetrizes the nonlinear part. Neglecting the higher order terms in $\epsilon$,
one obtain skew-adjoint perturbations of symmetric quasilinear hyperbolic systems and the classical theory of  this kind of systems
provides the  $O(\frac{1}{\epsilon})$ existence time. This method does not use the dispersive part of the systems and of course does
not solve the long time existence problem for the {\it abcd} systems.

In \cite{MSZ} an approach   based on a Nash-Moser theorem is developed  to prove well-posedness results  for the 2D- Boussinesq
systems \footnote{This approach gives also similar results in the one dimensional case.} on time intervals of order $1/\epsilon$.
In  \cite{MSZ}  the {\it generic case}, that is when $b>0, d>0, a<0, c<0$ and  the {\it BBM/BBM} case, that is when $b>0, d>0, a=c=0$
are emphasized but the approach could very likely be applied to other cases . This method does not use the dispersive part of the systems
and  a high regularity level in required on the initial data, with loss of regularity on the solution.

The aim of the present paper is to prove the well-posedness for the
most dispersive of the two-dimensional Boussinesq systems
\eqref{abcd} on time scales of order $\epsilon^{-1/2},$ using the
dispersive properties of the systems. This allows to consider
relatively rough initial data. We will in particular obtain {\it
uniformly bounded in $\epsilon$} solutions on time scales of order
$T/\sqrt{\epsilon}$ in suitable Sobolev spaces, achieving the
rigorous justification of those models on the corresponding time
scales (see \cite{AL}). Note however that we are not able in our
functional setting to reach the optimal time scales
$O(\frac{1}{\epsilon})$.

The heart of the paper concerns the most dispersive Boussinesq
system (the {\it KdV-KdV} system) where $a=c=1/6, b=d=0$ which is an
interesting two-dimensional extension of the Korteweg-de Vries
equation \footnote{There are relatively few physically relevant
systems of this form.} for which the local well-posedness is not a
simple matter. Our method deeply lies on dispersive estimates for
the underlying linear problem. In particular we establish new
Strichartz and maximal function estimates.

 For the other cases (which are less dispersive), at least when $d>0,$
 \footnote{ When $d=0$ it is unclear whether the local well-posedness can be obtained by elementary energy methods in the $2D$ case
 without imposing the irrotationality of the velocity.} the proofs are the extension to the two-dimensional case of those given by
 energy estimates in \cite{BCS2} for the one-dimensional case, keeping track of the $\epsilon$'s but we do not pursue this issue here.
 We will focus instead on the order $2$ Boussinesq systems which can be written as  systems of nonlinear nonlocal Schr\"{o}dinger type
 equations coupled by the nonlinear terms.
 Some of those  cases have been studied in \cite{CTA} but the  $\epsilon$ dependence of the existence time interval was not addressed
 here and the difficulty linked to the case $d=0$  was underestimated.

\subsection{Organization of the paper}
The paper is organized as follows.

Section 2 recall briefly some useful facts on the Boussinesq systems.

 In Section 3 we will consider the {\it KdV-KdV} Boussinesq system, and establish the  well-posedness in the two-dimensional case,
 on time intervals of order $ \epsilon^{-1/2}, $ after establishing various dispersive estimates for the solutions of  the linear part.

 In Section 4, we consider  the well-posedness of the remaining {\it strongly dispersive} cases ( Schr\"{o}dinger type) to prove
 that there are also well-posed on time intervals of order $ \epsilon^{-1/2}. $ There are essentially two cases, depending whether
 $b$ or $d$ vanishes.



\subsection{Notations}

For any positive numbers $a$ and $b$, the notation $a \lesssim b$
means that there exists a positive constant $c$ such that $a \le c
b$. $C$ or $c$ will also denote various positive constants
independent of $\epsilon$.

We denote the horizontal variables by $x$ when $n=1$ and by
$x=(x_1,x_2)$ when $n=2$. We will also denote by $\cdot$ the
euclidian scalar product of two vectors $x=(x_1,x_2)$ and
$y=(y_1,y_2)$ of $\mathbb R^2$, which is to say $x\cdot
y=x_1y_1+x_2y_2$, and by $|x|$ the euclidian norm of  $x=(x_1,x_2)$,
\textit{i.e.} $|x|=\sqrt{x_1^2+x_2^2}$.

We use the Fourier multiplier notation: $f(D)u$ is defined as
${\mathcal F}(f(D)u)(\xi)=f(\xi) \widehat{u}(\xi)$, where
$\mathcal{F}$ and $\widehat{\cdot}$ stand for the Fourier transform,
which is defined by
$${\mathcal F}(\phi)(\xi)=\widehat{\phi}(\xi)=\frac{1}{2\pi}\int_{\mathbb
R^2}\phi(x)e^{-ix\cdot\xi}dx,$$
for any function $\phi:\mathbb R^2 \rightarrow \mathbb C$.

For $s \in \mathbb R$, we define the Bessel and Riesz potentials of
order $-s$ by $\Lambda^s=(1+\vert D\vert^2)^{s/2}$ and $D^s=|D|^s$.
Moreover, $\sqrt{-\Delta}$ will denote the Fourier multiplier of
symbol $|\xi|$. Observe that $\sqrt{-\Delta}=D^1$.

Let $R_j$ be the Riesz transforms, defined \textit{via} Fourier
transform by $R_j\phi=\left( -i
\frac{\xi_j}{|\xi|}\widehat{\phi}\right)^{\vee},$ for  $j=1, \ 2.$

The divergence operator will be denoted by $\nabla\cdot$ or $\divg{}.$

We denote by $\|\cdot\|_{L^p}$ ($1\leq p\leq\infty$) the standard
norm of the Lebesgue spaces $L^p(\R^n)$ ($n=1,2$) and $(\cdot
,\cdot)$ the scalar product in $L^2$.

If $\bv=(v_1,v_2)^T\in L^2(\R^2)^2$, then we write
$\|\bv\|_{L^2}=\big(\|v_1\|_{L^2}^2+\|v_2\|_{L^2}^2\big)^{1/2}$.

If $\bv=(v_1,v_2)^T\in L^\infty(\R^2)^2$, then we write
$\|\bv\|_{L^{\infty}}=\|v_1\|_{L^{\infty}}+\|v_2\|_{L^{\infty}}$.

The standard notation $H^s(\R^n)$, or simply $H^s$ if the underlying
domain is clear from the context, is used for the $L^2$-based
Sobolev spaces; their norm is written $\|\cdot\|_{H^s}$.

Finally, if $u=u(x,t)$ is a  function defined in $\mathbb R^2 \times
[0,T]$, respectively in $\mathbb R^2 \times \mathbb R$ and $1 \le p,
\ q \le \infty$, we define the mixed space-time spaces $L^q_TL^p_x$,
respectively $L^q_tL^p_x$, by the norms
\begin{displaymath}
\|u\|_{L^q_TL^p_x} =\Big( \int_0^T\|u(\cdot,t)\|_{L^p}^qdt\Big)^{\frac1q},
\quad \text{resp.} \quad \|u\|_{L^q_tL^p_x} =\Big( \int_{\mathbb R}\|u(\cdot,t)\|_{L^p}^qdt\Big)^{\frac1q},
\end{displaymath}
and $L^p_xL^q_T$, respectively $L^p_xL^q_t$, by the norms
\begin{displaymath}
\|u\|_{L^p_xL^q_T} =\Big( \int_{\mathbb R^d}\|u(x,\cdot)\|_{L^q_T}^pdx\Big)^{\frac1p},
\quad \text{resp.} \quad \|u\|_{L^p_xL^q_t} =\Big( \int_{\mathbb R^d}\|u(x,\cdot)\|_{L^q}^pdx\Big)^{\frac1p}.
\end{displaymath}
\

\section{Classification of  Boussinesq systems}

The Boussinesq systems can be conveniently classified according to the  linearization at the null solutions which
display their dispersive properties \cite{BCS2}. More precisely, the dispersion matrix writes in Fourier variables,

 \begin{displaymath}
\widehat{A}(\xi_1,\xi_2)=i\begin{pmatrix} 0 &\frac{ \xi_1(1-\epsilon a |\xi|^2)}{1+\epsilon b|\xi|^2)} & \frac{\xi_2(1-\epsilon a |\xi^2|^2)}{1+\epsilon b|\xi|^2)} \\
                 \frac{\xi_1(1-\epsilon c|\xi|^2}{1+\epsilon d|\xi|^2)}  & 0 & 0 \\
                 \\
                 \frac{\xi_2(1-\epsilon c|\xi|^2}{1+\epsilon d|\xi|^2)}    & 0 & 0 \end{pmatrix}.
\end{displaymath}

\vspace{0.3cm}
The corresponding non zero eigenvalues are

$$ \lambda_{\pm}=\pm i |\xi|\left(\frac{(1-\epsilon a|\xi|^2)(1-\epsilon c|\xi|^2)}{(1+\epsilon d|\xi|^2)(1+\epsilon b|\xi|^2)}\right)^{\frac{1}{2}}.$$

\vspace{0.3cm}
Recall \cite{BCS1} that the well-posedness of the linearized Boussinesq system requires that $b\geq 0,\; d\geq 0$ and
$a\leq 0,\;c\leq 0,$ (or $a=c)$.

The order of $\lambda_{\pm}$ determine the strength of the dispersion in the Boussinesq systems, the more dispersive
one corresponding to  $b=d=0,$ the so called {\it KdV-KdV} case which will be studied in details in the next section.

By \lq\lq weakly dispersive" Boussinesq systems, we mean the case where $b>0$ and $d>0$, so that $\lambda_{\pm}$ have order
$1,\;0\; \text{or}\; -1.$

 This situation has been studied in \cite{DMS1} where it is established for instance in the {\it generic} case with
 $a<0$ and $c<0$ that the Cauchy problem  is well-posed in $C(\lbrack 0,T_{\epsilon}\rbrack; H^1(\R^2))^2$
 where $T_{\epsilon}$ is of order $O(\frac{1}{\epsilon^{\beta}})$ for any $\beta<\frac{1}{2}$.

 As already mentioned, existence on time intervals of order $O(\frac1{\epsilon})$ has been established for weakly
 dispersive Boussinesq systems when the initial data are smooth (and with a loss of derivatives) in \cite{MSZ}.

 \section{The KdV-KdV Boussinesq system}
We are interested in the following dispersive system
\begin{equation} \label{KdV type}
\left\{ \begin{array}{l}   \partial_t\eta+\text{div}\,
\textbf{v}+\epsilon\text{div} \, (\eta\textbf{v})+\epsilon\text{div}
\,
\Delta \textbf{v}=0 \\
\partial_t \textbf{v} +\nabla \eta +\epsilon\frac12 \nabla(|\textbf{v}|^2)+\epsilon\nabla \Delta
\eta=0 \end{array} \right., \quad (x_1,x_2) \in \mathbb R^2, \ t \in
\mathbb R,
\end{equation}
which corresponds to the case where $b=d=0$ and $a=c=\frac{1}{6}.$
We have scaled $a$ and $c$ to the value $1$. This system might not
be the best one as modeling and numerical purposes are concerned (in
particular the cubic dispersion terms induce serious numerical
difficulties). Nevertheless \eqref{KdV type} is a mathematically
interesting system since it can be written as a nonlocal
two-dimensional version of a KdV -type system. Previous results
concerned the one-dimensional version:

\begin{equation} \label{1D}
\left\{ \begin{array}{l}   \eta_t+v_x+\epsilon(\eta v)_x+\epsilon v_{xxx}=0 \\
v_t + \eta_x +\frac{\epsilon}{2} (v^2)_x+\epsilon \eta_{xxx}=0
 \end{array} \right., \quad x \in \mathbb R, \ t \in
\mathbb R.
\end{equation}

Actually, as noticed in \cite{BCS2}, the change of variable $\eta=u+w, \; v=u-w$ reduces \eqref{1D} to the following system:

\begin{equation} \label{1Ddiag}
\left\{ \begin{array}{l}    u_t+u_x+\epsilon u_{xxx}
+\epsilon\lbrack 2uu_x+(uw)_x\rbrack=0 \\
w_t -w_x-\epsilon w_{xxx}+\epsilon\lbrack -2ww_x+(uw)_x\rbrack =0
 \end{array} \right., 
\end{equation}
which is  a system  of KdV type with uncoupled (diagonal) linear
part.  Thus (see \cite{BCS2}) the Cauchy problem is easily seen to
be locally well-posed for initial data in $H ^s(\R)\times H ^s(\R),$
$s>\frac{3}{4}$ by the results in \cite{KPV1}, \cite{KPV2}. On the
other hand, as noticed in \cite{ST} Appendix A in a slightly
different context, a minor modification of Bourgain's method as used
in \cite{KPV3} allows to solve the Cauchy problem for \eqref{1Ddiag}
for data in $H^s(\R)\times H^s(\R)$ with $s>-\frac{3}{4} $. We refer
to \cite{BGK} for details. It is worth noticing that in \cite{BGK}
the question of the dependence of the existence time with respect to
$\epsilon$ is not considered.

Coming back to the two-dimensional system, we will establish that
the Cauchy problem is locally well-posed for data in
$H^s(\R^2)\times H^s(\R^2)^2$ where $s>\frac{3}{2}$.
\begin{theorem} \label{theoKdV-KdV}
Let $s>\frac32$ and $0 < \epsilon \le 1$ be fixed. Then for any
$(\eta_0,\textbf{v}_0) \in H^s(\mathbb R^2)\times H^s(\mathbb
R^2)^2$ with $\curl \textbf{v}_0= 0$, there exist a positive time
$T=T(\|(\eta_0,\textbf{v}_0)\|_{H^s\times (H^s)^2})$, a space
$Y^s_{T_{\epsilon}}$ such that
\begin{equation} \label{theoKdV-KdV1}
Y^s_{T_{\epsilon}} \hookrightarrow C\big([0,T_{\epsilon}] \ ; \ H^s(\mathbb R^2)\times H^s(\mathbb R^2)^2\big),
\end{equation}
and a unique solution $(\eta,\textbf{v})$ to \eqref{KdV type} in $Y^s_{T_{\epsilon}}$ satisfying
$(\eta,\textbf{v})_{|_{t=0}}=(\eta_0,\textbf{v}_0)$, where $T_{\epsilon}=T\epsilon^{-\frac12}$.

Moreover, for any $T' \in (0,T_{\epsilon})$, there exists a neighborhood $\Omega^s$ of $(\eta_0,\textbf{v}_0)$
in $H^s(\mathbb R^2)\times H^s(\mathbb R^2)^2$ such that the flow map associated to \eqref{KdV type} is smooth
from $\Omega^s$ into $Y^s_{T'}$.
\end{theorem}

\begin{remark}\label{pareil}
Our proof applies as well with minor modifications to Boussinesq systems with $b=d=0$ and $a<0,$ $c<0$. While this
case is excluded for systems modeling purely gravity waves, it may happen for  capillary-gravity waves when the surface
tension parameter  $\tau$ is large enough.
\end{remark}

\begin{remark}\label{uniform}
It transpires from the proof of Theorem \ref{theoKdV-KdV} that the solution is uniformly bounded (with respect to $\epsilon$)
in the corresponding spaces on a time interval $\lbrack 0, \frac{T}{\sqrt{\epsilon}}\rbrack$.
\end{remark}

The strategy is first to diagonalize the linear part of \eqref{KdV type}. We thus
reduce the original system to a {\it nonlocal} one. In particular,
the nonlinear part involves order zero pseudo-differential operators
(in fact Riesz transforms). We then solve the underlying Duhamel
integral formulation by a fixed point argument in a ball of a Banach
space  constructed from the various dispersive estimates satisfied
by the linear part. Note that here, we are able to compute the dependence on $\epsilon$ of the constants appearing in the
linear estimates. This allows us to obtain an existence result in a time interval $[0,T]$ depending on $\epsilon$, in our
case $T\sim \epsilon^{-\frac12}$.

We will thus proceed  as follows. We first derive the Hamiltonian
formulation of \eqref{KdV type}. We then perform the diagonalization
and  state the dispersive estimates which are to be used in the
fixed point argument. The maximal function and Strichartz estimates
seem to be new. Finally, we solve the new system \eqref{KdV type3}
which is equivalent to \eqref{KdV type}, assuming $\curl {\bf v}_0=
0$.

\subsection{Hamiltonian structure}

The system \eqref{KdV type} can be rewritten on the form
\begin{equation} \label{KdV type2}
\partial_t \textbf{u}+ A_{\epsilon}\textbf{u}+\epsilon\mathcal{N}(\textbf{u})=0,
\end{equation}
where
\begin{displaymath} \label{mat}
\textbf{u}=\begin{pmatrix} \eta \\ v_1 \\ v_2 \end{pmatrix}, \quad
A_{\epsilon}=\begin{pmatrix} 0 & (1+\epsilon\Delta)\partial_{x_1} & (1+\epsilon\Delta)\partial_{x_2} \\
                 (1+\epsilon\Delta)\partial_{x_1} & 0 & 0 \\
                 (1+\epsilon\Delta)\partial_{x_2} & 0 & 0 \end{pmatrix},
\end{displaymath}
and
\begin{displaymath}
\mathcal{N}(\textbf{u})=\begin{pmatrix}
\partial_{x_1}(\eta
v_1)+\partial_{x_2}(\eta v_2) \\ \frac12 \partial_{x_1}(v_1^2+v_2^2) \\
\frac12 \partial_{x_2}(v_1^2+v_2^2) \end{pmatrix}.
\end{displaymath}
We will denote by $(\cdot,\cdot)$ the scalar product on $L^2(\mathbb
R^2;\mathbb R^3)$, \textit{i.e}
\begin{displaymath}
(\textbf{u},\widetilde{\textbf{u}})=\int_{\mathbb R^2}\left(\eta
\widetilde{\eta}+v_1 \widetilde{v}_1+v_2 \widetilde{v}_2
\right)dx_1dx_2
\end{displaymath}
and by $J$ the skew adjoint matrix
\begin{displaymath}
J=\begin{pmatrix} 0 & \partial_{x_1} & \partial_{x_2} \\
                 \partial_{x_1} & 0 & 0 \\
                 \partial_{x_2} & 0 & 0 \end{pmatrix}.
\end{displaymath}
Then, the system \eqref{KdV type} is equivalent to
\begin{displaymath}
\partial_t \textbf{u}=-J \begin{pmatrix} (1+\epsilon\Delta)\eta+\frac{\epsilon}2 |\textbf{v}|^2\\
(1+\epsilon\Delta)v_1+\epsilon\eta v_1 \\
(1+\epsilon\Delta)v_2+\epsilon\eta v_2
\end{pmatrix}=J (\text{grad} \, H_{\epsilon})(\textbf{u}),
\end{displaymath}
where $H_{\epsilon}(\textbf{u})$ is the functional given by
\begin{displaymath}
H_{\epsilon}(\textbf{u})=\frac12 \int_{\mathbb R^2}\big(\epsilon|\nabla
\eta|^2+\epsilon|\nabla \textbf{v}|^2-\eta^2-|\textbf{v}|^2-\epsilon
\eta|\textbf{v}|^2\big)dx_1dx_2.
\end{displaymath}
Therefore, it follows that $H_{\epsilon}$ is a conserved quantity by the flow
of \eqref{KdV type}, since
\begin{displaymath}
\frac{d}{dt}H_{\epsilon}(\textbf{u})=H_{\epsilon}'(\textbf{u})\,\partial_t\textbf{u}
=\big((\text{grad}\, H_{\epsilon})(\textbf{u}),\partial_t\textbf{u}\big)=
\big((\text{grad} \, H_{\epsilon})(\textbf{u}),J(\text{grad} \,
H_{\epsilon})(\textbf{u})\big)=0,
\end{displaymath}
where we used the fact that $J$ is skew adjoint.

\subsection{Linear estimates}
\subsubsection{Diagonalization} \label{diagonalization}
We transform here \eqref{KdV type} into an equivalent system with a
{\it diagonal} linear part.

First, we observe that the Fourier transform of $A$ is given by
\begin{displaymath}
\widehat{A}_{\epsilon}(\xi_1,\xi_2)=\begin{pmatrix} 0 & i(1-\epsilon|\xi|^2)\xi_1 & i(1-\epsilon|\xi|^2)\xi_2 \\
                 i(1-\epsilon|\xi|^2)\xi_1  & 0 & 0 \\
                 i(1-\epsilon|\xi|^2)\xi_2  & 0 & 0 \end{pmatrix}.
\end{displaymath}
We compute the characteristic polynomial of
$\widehat{A}_{\epsilon}(\xi_1,\xi_2)$,
\begin{displaymath}
\chi_{\widehat{A}_{\epsilon}(\xi_1,\xi_2)}(\lambda)=-\lambda(\lambda^2+(1-\epsilon|\xi|^2)^2|\xi|^2),
\end{displaymath}
so that its eigenvalues are given by
\begin{displaymath}
\lambda_0=0, \quad \lambda_1=i(1-\epsilon|\xi|^2)|\xi|, \quad
\text{and} \quad \lambda_2=-i(1-\epsilon|\xi|^2)|\xi|,
\end{displaymath}
with associated eigenvectors
\begin{displaymath}
E_0=\begin{pmatrix} 0 \\ \frac{\xi_2}{|\xi|} \\ -\frac{\xi_1}{|\xi|}
\end{pmatrix}, \quad E_1=\begin{pmatrix} 1 \\ \frac{\xi_1}{|\xi|} \\ \frac{\xi_2}{|\xi|}
\end{pmatrix}, \quad \text{and} \quad E_2=\begin{pmatrix} -1 \\
\frac{\xi_1}{|\xi|} \\ \frac{\xi_2}{|\xi|}
\end{pmatrix}.
\end{displaymath}
Then, if we denote by $\widehat{P}$ the matrix of $(E_0,E_1,E_2)$ in
the canonical basis, we have that
\begin{displaymath}
\widehat{P}=\begin{pmatrix} 0 & 1 & -1 \\ \frac{\xi_2}{|\xi|} &
\frac{\xi_1}{|\xi|} & \frac{\xi_1}{|\xi|}
\\ -\frac{\xi_1}{|\xi|} &\frac{\xi_2}{|\xi|} & \frac{\xi_2}{|\xi|}
\end{pmatrix}, \quad \text{and} \quad
\widehat{P}^{-1}=\frac{1}{2}\begin{pmatrix} 0 & 2\frac{\xi_2}{|\xi|} & -2\frac{\xi_1}{|\xi|} \\
1 & \frac{\xi_1}{|\xi|} & \frac{\xi_2}{|\xi|} \\ -1
&\frac{\xi_1}{|\xi|} & \frac{\xi_2}{|\xi|}
\end{pmatrix}.
\end{displaymath}

Therefore, the linear part of \eqref{KdV type2} is equivalent to
$\partial_t\textbf{w}+D\textbf{w}=0$, where
\begin{displaymath}
\textbf{w}=\begin{pmatrix} w_0 \\ w_1 \\ w_2
\end{pmatrix}=P^{-1}\textbf{u}, \quad
P^{-1}=\frac{1}{2}\begin{pmatrix} 0 & 2iR_2 & -2iR_1 \\
1 & iR_1 & iR_2 \\ -1 &iR_1 & iR_2
\end{pmatrix},
\end{displaymath}
and
\begin{displaymath}
D=\begin{pmatrix} 0 & 0 & 0 \\ 0 & i(1+\epsilon\Delta)\sqrt{-\Delta} & 0 \\
0 & 0 & -i(1+\epsilon\Delta)\sqrt{-\Delta} \end{pmatrix}.
\end{displaymath}

Next, we turn to the nonlinear part of \eqref{KdV type2}.
$\mathcal{N}(\textbf{u})$ is given as function of $\textbf{w}$ by
\begin{displaymath}
\begin{pmatrix}
\partial{x_1}\left((w_1-w_2)(R_2w_0+R_1(w_1+w_2)) \right)+
\partial{x_2}\left((w_1-w_2)(-R_1w_0+R_2(w_1+w_2)) \right) \\
\frac12 \partial_{x_1}\left(
(R_2w_0+R_1w_1+R_1w_2)^2+(-R_1w_0+R_2w_1+R_2w_2)^2\right) \\
\frac12 \partial_{x_2}\left(
(R_2w_0+R_1w_1+R_1w_2)^2+(-R_1w_0+R_2w_1+R_2w_2)^2\right)
\end{pmatrix}.
\end{displaymath}
Then, deduce, using the identities
$R_2\partial_{x_1}=R_1\partial_{x_2}$ and
$R_1\partial_{x_1}+R_2\partial_{x_2}=\sqrt{-\Delta}$, that
\begin{displaymath}
P^{-1}\mathcal{N}(\textbf{u})=\frac12\begin{pmatrix} 0
\\ I +II\\ -I+II
\end{pmatrix}=: \widetilde{\mathcal{N}}(\textbf{w}),
\end{displaymath}
where
\begin{equation} \label{I}
\begin{split}
I=I(w_1,w_2)\:=&(w_1-w_2)\sqrt{-\Delta}(w_1+w_2)+\partial_{x_1}(w_1-w_2)R_1(w_1+w_2) \\&
+\partial_{x_2}(w_1-w_2)R_2(w_1+w_2),
\end{split}
\end{equation}
and
\begin{equation} \label{II}
II=II(w_1,w_2):=i\sqrt{-\Delta}\left(
(R_1(w_1+w_2))^2+(R_2(w_1+w_2))^2\right).
\end{equation}

Summarizing we have that \eqref{KdV type} is equivalent to
\begin{equation} \label{KdV type3}
\left\{\begin{array}{ll}
\partial_tw_1+ i(1+\epsilon\Delta)\sqrt{-\Delta}w_1+(I+II)(w_1,w_2)=0 \\
\partial_tw_2- i(1+\epsilon\Delta)\sqrt{-\Delta}w_2+(-I+II)(w_1,w_2)=0
\end{array}\right. ,
\end{equation}
where $I(w_1,w_2)$ and $II(w_1,w_2)$ are defined in \eqref{I} and
\eqref{II}.

\begin{remark}
Note that we use in our analysis that $w_0=0$. Indeed, the equation
on $w_0$ is $\partial_tw_0=0$. Moreover,
\begin{displaymath}
w_0=0 \quad \iff \quad R_2v_1=R_1v_2 \quad \iff \text{curl} \,
\textbf{v}=0.
\end{displaymath}
We observe that this condition is physically relevant. The Boussinesq systems
derived from the water waves equations, where the fluid is supposed
to be irrotational and ${\bf v}$ is an $O(\epsilon^2)$ approximation of the horizontal velocity at a certain depth
which is a gradient. Note also that since the equation for ${\bf v}$ writes $\partial_t {\bf v}= \nabla F$, the condition
$curl {\bf v} =0$ is preserved by the evolution.
\end{remark}

Now, to derive the smoothing effects associated to the linear part
of \eqref{KdV type}, it suffices to consider the linear
system
\begin{equation}\label{suff}
\left\lbrace
  \begin{array}{l}
u_t\pm L_{\epsilon}u=0,\\
u(.,0)= u_0,
\end{array}\right.
\end{equation}
where
$L_{\epsilon}=i(I+\epsilon\Delta)\sqrt{-\Delta}=-i\varphi_{\epsilon}(D)$,
$\textit{i.e.}$ $\varphi_{\epsilon}(\xi)=\epsilon|\xi|^3-|\xi|$. In
the sequel, we will identify
$\varphi_{\epsilon}(\xi)=\varphi_{\epsilon}(|\xi|)$. Moreover, we
will denote by $U_{\epsilon}^{\pm}(t)u_0$ the solution of
\eqref{suff}, \textit{i.e.}
\begin{equation} \label{group}
U_{\epsilon}^{\pm}(t)u_0=\Big(e^{\pm i
t\varphi_{\epsilon}}\widehat{u_0} \Big)^{\vee}.
\end{equation}
The smoothing effects are of three different types.

\subsubsection{Dispersive smoothing estimates}

We will use the results in \cite{KPV} (see also the general results
in \cite{KS}) to deduce the following {\it local smoothing
estimates}.

Let $\big\{Q_{\alpha}\big\}_{\alpha \in \mathbb Z^2}$ denote a family
of nonoverlapping cubes of unit size such that $\mathbb
R^2=\underset{\alpha \in \mathbb Z^2}{\cup}Q_{\alpha}$.
\begin{theorem}\label{locsm}
Let $T>0$ and $\epsilon>0$. Then, it holds that
\begin{equation} \label{locsm.1}
\sup_{\alpha}\Big(\int_{Q_{\alpha}}\int_0^T
\big|P_{>\epsilon^{-\frac12}}D^1_xU_{\epsilon}^{\pm}(t)u_0(x)\big|^2dtdx\Big)^{\frac12}
\lesssim \epsilon^{-\frac12}\|u_0\|_{L^2_x},
\end{equation}
where the implicit constant does not depend on $\epsilon$ and $T$.
\end{theorem}

\begin{proof} Without loss of generality, we can assume that $Q_{\alpha}=Q:=\{x \ : \
|x|<1\}$. We fix a function $\eta$ in $C_0^{\infty}({\mathbb R})$
such that $0\le \eta \le 1$, $\eta \equiv 1$ on $[-1,1]$ and
$\text{supp}\, \eta \subset [0,2]$. Let us define
$\eta_{\epsilon}(x)=\epsilon^{\frac12}\eta(\epsilon^{\frac12}x)$.
Then $P_{\le \epsilon^{-\frac12}}$, respectively
$P_{>\epsilon^{-\frac12}}$ denote the operators defined by
\begin{equation} \label{P}
P_{\le
\epsilon^{-\frac12}}u_0=\mathcal{F}^{-1}\big(\eta_{\epsilon}(|\cdot|)\widehat{u}_0\big)
\quad \text{and} \quad
P_{>\epsilon^{-\frac12}}=1-P_{\le\epsilon^{-\frac12}}.
\end{equation}
Observe that in the support of
$1-\eta_{\epsilon}(|\cdot|)$, we have that $|\nabla
\varphi_{\epsilon}(\xi)|=3\epsilon|\xi|^2-1>0$.
Then it follows by using Theorem 4.1, formula (4.2), p 54-55 in
\cite{KPV} that
\begin{equation} \label{locsm.3}
\begin{split}
\big\|D^1_xU_{\epsilon}^{\pm}(t)P_{>
\epsilon^{-\frac12}}u_0\big\|_{L^2_{Q\times[0,T]}} &\lesssim
\Big(\int_{|\xi|>\epsilon^{-\frac12}}
\frac{|\xi|^2}{3\epsilon|\xi|^2-1}|\widehat{u}_0(\xi)|^2d\xi\Big)^{\frac12} \\ &\lesssim
\epsilon^{-\frac12}\|u_0\|_{L^2_x},
\end{split}
\end{equation}
which concludes the proof of Theorem \ref{locsm}.
\end{proof}

\subsubsection{The maximal function estimate}

We will prove our maximal function estimate in the $n$-dimensional
case. In other words, we will consider the unitary group
$U_{\epsilon}^{\pm}(t)=e^{\pm i t\varphi_{\epsilon}(D)}$, where
$\varphi_{\epsilon}(\xi)=\varphi_{\epsilon}(|\xi|)=\epsilon|\xi|^3-|\xi|$
and $\xi \in \mathbb R^n$, for $n \ge 2$. Let
$\{Q_{\alpha}\}_{\alpha \in \mathbb Z^n}$ denote the mesh of dyadic
cubes of unit size. The main result of this subsection reads as
follows.
\begin{theorem} \label{maximal}
With the above notation, for any $s>\frac{3n}4$, $\epsilon>0$ and $T>0$ satisfying
$\epsilon T \le 1$, it holds that
\begin{equation} \label{maximal.1}
\Big( \sum_{\alpha \in \mathbb Z^n} \sup_{|t| \le T}\sup_{x \in
Q_{\alpha}}\big|U_{\epsilon}^{\pm}(t)u_0(x) \big|^2\Big)^{\frac12} \lesssim
(1+T^{\frac{n}4-\frac14})\|u_0\|_{H^s},
\end{equation}
where the implicit constant does not depend on $\epsilon$ and $T$.
\end{theorem}

We will only treat the case of $U_+^{\epsilon}$, since the case of $U_-^{\epsilon}$ is similar. The proof of Theorem \ref{maximal} is based on the next lemma.

\begin{lemma} \label{lemma maximal}
For $k \in \mathbb Z_+$, let $\psi_k \in C_0^{\infty}([2^{k-1},2^{k+1}])$ be such that $0 \le
\psi_k \le 1$. Then for  $\epsilon t \in (0,2]$,
\begin{equation} \label{lemma maximal.1}
\Big|\int_{\mathbb R^n}e^{i(t\varphi_{\epsilon}(\xi)+x.\xi)} \psi_k(|\xi|)d\xi
\Big| \le cH_k(|x|),
\end{equation}
where $H_k$ is decreasing and satisfies
\begin{equation} \label{lemma maximal.2}
\int_{\mathbb R^n}H_k(|x|)dx \lesssim (1+|t|^{\frac{n}2-\frac12})2^{\frac{3kn}2},
\end{equation}
and $H_k(r) \le c2^{\frac{3kn}2}$ for $r \in (0,10)$. Also, a
similar result holds for $\psi \in C_0^{\infty}([-10,10])$ with
$2^{\frac{3kn}2}$ replaced by $c$. Observe that the implicit constant does not depend on $\epsilon$ and $t$, but may depend on the dimension $n$.
\end{lemma}

The following estimate, essentially proved in Proposition 2.6 of
\cite{KPV1}, will be useful.
\begin{proposition} \label{prop maximal}
For $k \in \mathbb Z_+$, $\epsilon t \in (0,2]$, $r \in \mathbb R$
and $\psi_k$ as in Lemma \ref{lemma maximal}, define
\begin{equation} \label{prop maximal.1}
I_k(t,r)=\int_0^{+\infty}e^{i(\epsilon ts^3+sr)}\psi_k(s)ds.
\end{equation}
Then
\begin{equation} \label{prop maximal.2}
|I_k(t,r)| \le F_k(r)= \left\{\begin{array}{lll}
c2^k  & \text{for} & |r| \le 1 \\
c2^{\frac{k}{2}}r^{-\frac12} & \text{for} & 1 \le |r| \le c2^{2k} \\
cr^{-N} & \text{for} & |r|>c2^{2k}
\end{array} \right.
\end{equation}
for any $N \in \mathbb Z_+$.
\end{proposition}

Moreover, it is known (see \cite{SW} for example) that the Fourier
transform of a radial function $f(|x|)=f(s)$ is still radial and is
given by
\begin{equation} \label{lemma maximal.3}
\widehat{f}(r)=\widehat{f}(|\xi|)=r^{-\frac{n-2}2}
\int_0^{\infty}f(s)J_{\frac{n-2}2}(rs)s^{\frac{n}{2}}ds,
\end{equation}
where $J_m$ is the Bessel function, defined by
\begin{equation} \label{lemma maximal.3b}
J_m(r)=\frac{(r/2)^m}{\Gamma(m+1/2)\pi^{\frac12}}\int_{-1}^1e^{irs}(1-s^2)^{m-\frac12}ds,
\quad \text{for} \quad m>-\frac12.
\end{equation}
Next, we list some properties of the Bessel functions (see
\cite{SW}, \cite{KPV5}, \cite{GPW} and the references therein).
\begin{lemma} \label{lemma maximal2} It holds that
\begin{equation} \label{lemma maximal2.1}
J_m(r)\underset{r \to 0}{=}O(r^m),
\end{equation}
\begin{equation} \label{lemma maximal2.2}
J_m(r)\underset{r \to
+\infty}{=}e^{-ir}\sum_{j=0}^N\alpha_{m,j}r^{-(j+\frac12)}
+e^{ir}\sum_{j=0}^N\tilde{\alpha}_{m,j}r^{-(j+\frac12)}+O(r^{-(N+\frac32)}),
\end{equation}
for any $N \in \mathbb Z_+$, and
\begin{equation} \label{lemma maximal2.3}
r^{-\frac{n-2}2}J_{\frac{n-2}2}(r)=c_n\mathcal{R}\big(e^{ir}h(r)\big),
\end{equation}
where $h$ is a smooth function satisfying
\begin{equation} \label{lemma maximal2.4}
\big|\partial_r^kh(r)\big| \le c_k(1+r)^{-\frac{n-1}2-k},
\end{equation}
for any $k \in \mathbb Z_+$.
\end{lemma}

\begin{proof}[Proof of Lemma \ref{lemma maximal}]
By using \eqref{lemma maximal.3}, it follows that
\begin{equation} \label{lemma maximal.4}
\mathcal{I}_k(t,r):=\int_{\mathbb
R^n}e^{i(t\varphi_{\epsilon}(\xi)+x.\xi)}\psi_k(|\xi|)d\xi
=r^{-\frac{n-2}{2}}\int_0^{+\infty}e^{it(\epsilon s^3-s)}\psi_k(s)
J_{\frac{n-2}2}(rs)s^{\frac{n}{2}}ds,
\end{equation}
where $r=|x| \in [0,+\infty)$ and $\epsilon t \in [0,2]$.

When $0 \le |r-t| \le 1$ or $0\le r \le 1$, it follows from \eqref{lemma maximal2.1} (or \eqref{lemma maximal2.3}--\eqref{lemma maximal2.4}) that
\begin{equation} \label{lemma maximal.5}
\big|\mathcal{I}_k(t,r)\big| \lesssim
\int_0^{+\infty}\psi_k(s)s^{n-1}ds=c2^{nk}.
\end{equation}

In the case $r>1$ and $|r-t|>1$, we substitute $J_{\frac{n-2}{2}}$ by the
right-hand side of \eqref{lemma maximal2.2} in \eqref{lemma
maximal.4} and evaluate successively each term of the sum and the
remainder. Here we consider only the most difficult case, when $\tilde{\alpha}_{m,j}=0$. Then, the $j^{th}$ term has the form
\begin{displaymath}
\mathcal{I}_{k,j}(t,r):=r^{-\frac{n-2}{2}}\int_0^{+\infty}e^{i(t(\epsilon s^3-s)+sr)}
\psi_k(s)(rs)^{-(j+\frac12)}s^{\frac{n}2}ds,
\end{displaymath}
so that
\begin{displaymath}
\big|\mathcal{I}_{k,j}(t,r) \big| \le
r^{-\frac{n-2}{2}}(2^kr)^{-(j+\frac12)}2^{k\frac{n}2}
\Big|\int_0^{+\infty}e^{i(t\epsilon s^3+s(r-t)}\tilde{\psi}_k(s)ds \Big|,
\end{displaymath}
where $\tilde{\psi}_k$ is another function satisfying
$\tilde{\psi}_k \in C_0^{\infty}([2^{k-1},2^{k+1}])$ and $0 \le
\tilde{\psi}_k \le 1$. Therefore, we deduce from Proposition
\ref{prop maximal} that
\begin{equation} \label{lemma maximal.5a}
\big|\mathcal{I}_{k,j}(t,r) \big| \lesssim \left\{
\begin{array}{lll}2^{k(\frac{n}2-j)}r^{-(\frac{n}2-\frac12+j)}|r-t|^{-\frac12} & \text{for} & |r-t| \in [1,c2^{2k}] \\
2^{k(\frac{n}2-j-\frac12)}r^{-(\frac{n}2-\frac12+j)}|r-t|^{-m} & \text{for} & |r-t| > c2^{2k},\end{array} \right.
\end{equation}
for any $m \in \mathbb Z_+$. Next, we fix $N=N(n)>\frac{n-1}2$ and
bound the remainder
\begin{displaymath}
\mathcal{R}_k(t,r):=r^{-\frac{n-2}2}\int_0^{+\infty}e^{i(t\varphi_{\epsilon}(s)+sr)}
\psi_k(s)(rs)^{-(N+\frac32)}s^{\frac{n}2}ds
\end{displaymath}
as follows
\begin{equation} \label{lemma maximal.6}
\big|\mathcal{R}_k(t,r) \big| \lesssim
2^{k(\frac{n}2-N-\frac12)}r^{-(\frac{n}2+N+\frac12)} \lesssim
r^{-m},
\end{equation}
with $m>n$.

Therefore, if we define
\begin{displaymath}
\mathcal{H}_k(|x|)=\left\{\begin{array}{lll}2^{kn} & \text{for} & |x| \le 1 \ \text{or} \
||x|-t| \le 1\\
 \sum_{j=0}^N2^{k(\frac{n}2-j)}|x|^{-(\frac{n}2-\frac12+j)}||x|-t|^{-\frac12}
 & \text{for} &
 1 \le ||x|-t| \le c2^{2k} \\ 2^{k(\frac{n}2-\frac12)}||x|-t|^{-m} & \text{for} & ||x|-t| >c2^{2k},
 \end{array} \right.
\end{displaymath}
and
\begin{displaymath}
H_k(|x|)=\mathcal{H}_k(|x|)+\frac{2^{kn}}{(1+|x|)^m}
\end{displaymath}
it follows from \eqref{lemma maximal.4}--\eqref{lemma maximal.6}
that $\big|\mathcal{I}_k(t,|x|)\big|\le cH_k(|x|)$ and a simple
computation leads to
\begin{displaymath}
\begin{split}
\int_{\mathbb R^n}H_k(|x|)dx &\le c2^{kn}+c2^{\frac{kn}2}
\int_{1 \le |r-t|\le c2^{2k}}r^{\frac{n}2-\frac12}|r-t|^{-\frac12}dr \\
& \le c2^{\frac{3kn}2}(1+|t|^{\frac{n}2-\frac12}),
\end{split}
\end{displaymath}
which concludes the proof of Lemma \ref{lemma maximal}.
\end{proof}

Finally, at this point, the proof of Theorem \ref{maximal} follows
closely the argument of Kenig, Ponce and Vega in the case of the
Schr\"odinger equation in Theorem 3.2 of \cite{KPV5}. Therefore, we
will omit it.

\subsubsection{Strichartz estimates}

Strichartz estimates for unitary groups of the form $e^{it\phi(D)}$
in $\mathbb R^n$, $n \ge 2$, were derived in the case where $\phi$
is an elliptic polynomial by Kenig, Ponce and Vega \cite{KPV} and in
the case where $\phi$ is a general polynomial in $\mathbb R^2$ by
Ben-Artzi, Koch and Saut \cite{BAKS}. When the phase function $\phi$
is a radial (nonhomogeneous) function and its derivative does not
vanish, some techniques were recently developed by Cho and Ozawa
\cite{CO} and Guo, Peng and Wang \cite{GPW}.

In the sequel, we will use the techniques developed in \cite{GPW},
based on the ones used in \cite{KPV}  and on the representation of the Fourier
transform of a radial function in terms of the Bessel function (see
formula \eqref{lemma maximal.3}), to prove Strichartz estimates
associated to the unitary groups $U^{\pm}_{\epsilon}$ defined in
\eqref{group}. However, in our case, we do not need to perform a
dyadic decomposition in frequencies.
\begin{theorem} \label{Strichartz-theo}
Let $0<\epsilon \le 1$, $T>0$  and $0 \le \alpha < \frac12$. Then,
it holds that
\begin{equation} \label{Strichartz-theo.1}
\|D^{\alpha}_xU_{\epsilon}^{\pm}u_0\|_{L^{q_{\alpha}}_TL^{\infty}_x}
\lesssim \epsilon^{-\kappa_{\alpha}}\|u_0\|_{L^2},
\end{equation}
for all $u_0 \in L^2(\mathbb R^2)$, where the implicit constant is
independent of $\epsilon$ and $T$, $q_{\alpha}$ is the root of the
polynomial
$$3q^2-2(7-2\alpha)q+12=0,$$
satisfying $q_{\alpha}>2$ and
$\kappa_{\alpha}=\frac12+\frac{\alpha}2-\frac1{4q_{\alpha}}$.
\end{theorem}

\begin{remark} When $\alpha=0$, then
$q=\frac{7+\sqrt{13}}3=\frac72+$ and
$\kappa=\frac12-\frac1{4q}=\frac37+$. On the other hand, we have
that $\lim_{\alpha \rightarrow \frac12}q_{\alpha}=2$ and
$\lim_{\alpha \rightarrow \frac12}\kappa_{\alpha}=\frac58$
\end{remark}

For sake of simplicity, we will fix $U_{\epsilon}=U^+_{\epsilon}$ in
the rest of this subsection. First, we derive the following decay
estimate for the solution to the linear problem \eqref{suff}.
\begin{proposition} \label{Strichartz-prop}
Let $0<\epsilon \le 1$ and $0\le \beta \le 1$. Then, it holds that
\begin{equation} \label{Strichartz-prop1}
\|D_x^{\beta}U_{\epsilon}(t)u_0\|_{L^{\infty}_x} \lesssim k_{\beta,
\epsilon}(t)\|u_0\|_{L^1},
\end{equation}
for all $t \in \mathbb R$, where $k_{\beta, \epsilon}$ is given by
\begin{equation} \label{Strichartz-prop1b}
k_{\beta, \epsilon}(t) = \left\{
\begin{array}{lll}
(\epsilon t)^{-\frac{2+\beta}3} & \text{if} & t\le \theta
\epsilon^{\frac12} \\
\epsilon^{-\frac34-\frac{\beta}2}t^{-\frac12} & \text{if} & t\ge
\theta \epsilon^{\frac12}
\end{array} \right. ,
\end{equation}
and $\theta$ is any positive constant independent of $\epsilon$.
\end{proposition}

The proof of Proposition \ref{Strichartz-prop} is based on formula
\eqref{lemma maximal.3}, Lemma \ref{lemma maximal2} and Van der
Corput's lemma:
\begin{lemma} \label{VdC}
Suppose that $f$ is a real valued $C^2$-function defined in $[a,b]$ such that
$|f''(\xi)|>1$ for any $\xi \in [a,b]$. Then
\begin{displaymath}
\Big|\int_a^be^{i\lambda f(\xi)} \psi(\xi)d\xi\Big|\lesssim |\lambda|^{-\frac12}
\big(\|\psi\|_{L^{\infty}}+\|\psi'\|_{L^1} \big),
\end{displaymath}
where the implicit constant does not depend on $a$ and $b$.
\end{lemma}

\begin{proof}[Proof of Proposition \ref{Strichartz-prop}]
First observe that
\begin{equation} \label{Strichartz-prop2}
\|D^{\beta}_xU_{\epsilon}(t)u_0\|_{L^{\infty}_x} \le \big\|
\big(|\cdot|^{\beta}e^{-it\varphi_{\epsilon}}
\big)^{\vee}\big\|_{L^{\infty}_x}\|u_0\|_{L^1}.
\end{equation}
On the other hand, formulas \eqref{lemma maximal.3} and \eqref{lemma
maximal2.3} imply that
\begin{equation} \label{Strichartz-prop4}
\begin{split}
\big(|\xi|^{\beta}e^{-it\varphi_{\epsilon}} \big)^{\vee}\big(x)
&=\int_0^{+\infty}s^{\beta}e^{it(\epsilon s^3-s)}J_0(rs)sds \\
&=\int_0^{+\infty}s^{\beta}e^{it(\epsilon s^3-s)}e^{irs}h(rs)sds+
\int_0^{+\infty}s^{\beta}e^{it(\epsilon s^3-s)}e^{-irs}\overline{h(rs)}sds\\
& =I_{\beta}(r)+II_{\beta}(r)
\end{split}
\end{equation}
where $r=|x|$.

For sake of simplicity, we will assume that $t \ge 0$ and only deal with $II_{\beta}$ since $I_{\beta}$ can be handled by similar techniques. We change variables $u=(\epsilon t)^{\frac13}s$ and deduce that
\begin{equation} \label{Strichartz-prop5}
II_{\beta}(r)=(\epsilon t)^{-\frac{2+\beta}{3}}A_{\beta}\big(\frac{r}{(\epsilon t)^{\frac13}}\big),
\end{equation}
where
\begin{displaymath}
A_{\beta}(r)=\int_0^{+\infty}e^{if_r(u)}\overline{h(ru)}u^{\beta+1}du,
\end{displaymath}
the phase function $f_r(u)$ is given by
\begin{displaymath}
f_r(u)=u^3-(r+\alpha)u, \quad \text{and} \quad \alpha=t^{\frac23}\epsilon^{-\frac13}.
\end{displaymath}
Therefore
\begin{equation} \label{Strichartz-prop6}
\sup_{r \ge 0} \big|A_{\beta}(r) \big| \lesssim
\max\{1,\alpha^{\frac{\beta}2+\frac14}\}=\max\{1,t^{\frac{\beta}3+\frac16}\epsilon^{-\frac{\beta}6-\frac{1}{12}}\},
\quad \forall \, \beta \in [0,1],
\end{equation}
would imply formula \eqref{Strichartz-prop1}.

To prove \eqref{Strichartz-prop6}, we introduce the smooth real-values functions $(\psi_1,\psi_2) \in C_0^{\infty}\times C^{\infty}$ such that $0\le \psi_1, \ \psi_2\le 1$, $\psi_1(u)+\psi_2(u)=1$,
\begin{displaymath}
\text{supp}\,\psi_1 \subset \big\{u \ : \ |3u^2-(r+\alpha)| \le \frac{r+\alpha}{2} \big\}
\end{displaymath}
and
\begin{displaymath}
\psi_2 =0 \quad \text{in} \quad \big\{u \ : \ |3u^2-(r+\alpha)| \le \frac{r+\alpha}{3} \big\}.
\end{displaymath}
It follows that
\begin{equation} \label{Strichartz-prop7}
\big|A_{\beta}(r)\big| \le \big|A^1_{\beta}(r)\big|+\big|A^2_{\beta}(r)\big|,
\end{equation}
where
\begin{displaymath}
A^j_{\beta}(r)=\int_0^{+\infty}e^{if_r(u)}\overline{h(ru)}u^{\beta+1}\psi_j(u)du, \quad j \in \{1,2\}.
\end{displaymath}

First, we deal with $A_{\beta}^2$. Observe that we can restrict the oscillatory integral in the range $u \in [1,+\infty]$, since otherwise when $u \in [0,1]$, the estimate is trivial. Moreover, we deduce from the triangle inequality that in the support of $\psi_2$, the derivative of the phase function satisfies $|f'_r(u)|=|3u^2-(r+\alpha)|>\frac16(u^2+(r+\alpha))$. Then, we obtain integrating by parts that
\begin{displaymath}
\begin{split}
A_{\beta}^2(r)&=\int_1^{+\infty}\frac{1}{if_r'(u)}\frac{d}{du}\big(e^{if_r(u)} \big)
\overline{h(ru)}u^{\beta+1}\psi_2(u)du \\
&=i\int_1^{+\infty}e^{if_r(u)}\frac{d}{du}\Big(
\frac{\overline{h(ru)}u^{\beta+1}\psi_2(u)}{f_r'(u)}\Big)
du.
\end{split}
\end{displaymath}
Therefore, it follows from \eqref{lemma maximal2.4} that
\begin{equation} \label{Strichartz-prop8}
|A_{\beta}^2(r)| \lesssim \int_1^{+\infty}\Big(\frac{ru^{\beta+1}}
{(u^2+(r+\alpha))(1+ru)^{\frac32}}
+\frac{u^{\beta+2}}{(u^2+(r+\alpha))^2(1+ru)^{\frac12}} \Big)du \lesssim 1.
\end{equation}
Note that the implicit constant does not depend on $r$, $\epsilon$ or $t$.

Next we turn to $A_{\beta}^1$. In the support of $\psi_1$, we have $u \sim (r+\alpha)^{\frac12}$,  so that $|f_r''(u)|=6u \gtrsim (r+\alpha)^{\frac12}$. Thus, it follows from Van der Corput's lemma that
\begin{equation} \label{Strichartz-prop9}
\begin{split}
|A_{\beta}^1(r)| &\lesssim \frac{1}{(r+\alpha)^{\frac14}}\Big(\|h(r\cdot)(\cdot)^{\beta+1}\psi_1\|_{L^{\infty}}
+\|\frac{d}{du}\big(h(r\cdot)(\cdot)^{\beta+1}\psi_1 \big)\|_{L^1} \Big)\\
& \lesssim
\frac{(r+\alpha)^{\frac12(1+\beta)}}{(r+\alpha)^{\frac14}(1+r(r+\alpha)^{\frac12})^{\frac12}}
\lesssim \max\{1,\alpha^{\frac{\beta}{2}+\frac14}\}.
\end{split}
\end{equation}

Finally, we deduce formula \eqref{Strichartz-prop6} combining \eqref{Strichartz-prop7}--\eqref{Strichartz-prop9}, which concludes the proof of Proposition \ref{Strichartz-prop}.
\end{proof}

We are now in position to give a proof of Theorem
\ref{Strichartz-theo}.

\begin{proof}[Proof of Theorem \ref{Strichartz-theo}] Fix $0 \le
\alpha <\frac12$, $\beta=2\alpha \in [0,1)$,
$\kappa=\kappa_{\alpha}$, $q=q_{\alpha}$ and $q'$ its conjugate
exponent, \text{i.e.} $\frac1q+\frac1{q'}=1$. We first observe by
using a P. Tomas' duality argument (see for example \cite{LP}) that
estimate \eqref{Strichartz-theo.1} is equivalent to
\begin{equation} \label{Strichartz-theo.2}
\big\|\int_{-\infty}^{+\infty}D^{\beta}_xU_{\epsilon}(t-t')g(\cdot,
t')dt'\big\|_{L^{q}_tL^{\infty}_x} \lesssim
\epsilon^{-2\kappa}\|g\|_{L^{q'}_tL^1_x},
\end{equation}
for all $g \in L^{q'}(\mathbb R; L^1(\mathbb R^2))$.

Next we prove estimate \eqref{Strichartz-theo.2}. It follows from
Minkowski's inequality and estimate \eqref{Strichartz-prop1} that
\begin{equation} \label{Strichartz-theo.3}
\big\|\int_{-\infty}^{+\infty}D^{\beta}_xU_{\epsilon}(t-t')g(\cdot,
t')dt'\big\|_{L^{\infty}_x} \lesssim k_{\beta, \epsilon} \ast
\|g(\cdot)\|_{L^1_x}(t):=J_{\beta,\epsilon}(t),
\end{equation}
where $k_{\beta, \epsilon}$ is defined in \eqref{Strichartz-prop1b}.
We will denote $\varphi(t)=\|g(\cdot,t)\|_{L^1_x}$. Then we divide
the kernel $k_{\beta, \epsilon}$ in two parts, $k_{\beta,
\epsilon}=k_{\beta, \epsilon}^0+k_{\beta, \epsilon}^{\infty}$, where
\begin{displaymath}
k_{\beta, \epsilon}^0(t)=(\epsilon
t)^{-\frac{2+\beta}3}\chi_{\{|t|\le \theta \epsilon^{\frac12}\}}
\quad \text{and} \quad k_{\beta,
\epsilon}^{\infty}(t)=\epsilon^{-\frac34-\frac{\beta}2}t^{-\frac12}
\chi_{\{|t|\ge \theta \epsilon^{\frac12}\}},
\end{displaymath}
so that
\begin{equation} \label{Strichartz-theo.4}
J_{\beta,\epsilon}(t)=J_{\beta,\epsilon}^0(t)+J_{\beta,\epsilon}^{\infty}(t),
\end{equation}
where $J_{\beta,\epsilon}^0$, respectively
$J_{\beta,\epsilon}^{\infty}$, is the convolution operator
associated to the kernel $k_{\beta, \epsilon}^0$, respectively
$k_{\beta, \epsilon}^{\infty}$.

To estimate $J_{\beta,\epsilon}^0$, we observe that
\begin{displaymath}
\int_{\mathbb R}k_{\beta, \epsilon}^0(t)dt =2\int_0^{\theta
\epsilon^{\frac12}}(\epsilon
t)^{-\frac{2+\beta}3}dt=c\epsilon^{\frac{1+\beta}2}\theta^{\frac{1-\beta}3},
\end{displaymath}
since $0 \le \beta <1$. Therefore, it follows from Theorem 2 in Chapter III of
\cite{St} that
\begin{equation} \label{Strichartz-theo.5}
|J_{\beta,\epsilon}^0(t)| \lesssim
\epsilon^{\frac{1+\beta}2}\theta^{\frac{1-\beta}3}\mathcal{M}\varphi(t),
\end{equation}
where $\mathcal{M}$ denotes the Hardy-Littlewood maximal function.
On the other hand, Young's theorem implies that
\begin{equation} \label{Strichartz-theo.6}
|J_{\beta,\epsilon}^{\infty}(t)| \le
\left(\int_{|t|\ge\theta\epsilon^{\frac12}}t^{-\frac{q}2}dt
\right)^{\frac1q}\|\varphi\|_{L^{q'}}=c\epsilon^{-(1+\frac{\beta}2-\frac1{2q})}\theta^{\frac1q-\frac12}
\|\varphi\|_{L^{q'}},
\end{equation}
since $q>2$. Observe that
$1+\frac{\beta}2-\frac1{2q}>\frac{1+\beta}2$. Thus, we deduce
gathering \eqref{Strichartz-theo.4}--\eqref{Strichartz-theo.6} that
\begin{equation} \label{Strichartz-theo.6}
|J_{\beta,\epsilon}(t)| \lesssim
\epsilon^{-(1+\frac{\beta}2-\frac1{2q})}\big(\theta^{\frac1q-\frac12}
\|\varphi\|_{L^{q'}}+\theta^{\frac{1-\beta}3}\mathcal{M}\varphi(t)\big).
\end{equation}
Now, we choose $\theta=\theta(t)$ to minimize the term on the
right-hand side of \eqref{Strichartz-theo.6}, which is to say
\begin{displaymath}
\theta(t)^{\frac56-\frac{\beta}3-\frac1q}=\|\varphi\|_{L^{q'}}\mathcal{M}\varphi(t)^{-1}.
\end{displaymath}
This implies together with \eqref{Strichartz-theo.6} that
\begin{equation} \label{Strichartz-theo.7}
|J_{\beta,\epsilon}(t)| \lesssim
\epsilon^{-(1+\frac{\beta}2-\frac1{2q})}
\|\varphi\|_{L^{q'}}^{\gamma}\mathcal{M}\varphi(t)^{1-\gamma},
\end{equation}
where $\gamma=\frac{(1-\beta)2q}{5q-2\beta q-6}$. Then it follows
that
\begin{equation} \label{Strichartz-theo.8}
\|J_{\beta,\epsilon}\|_{L^q} \lesssim
\epsilon^{-(1+\frac{\beta}2-\frac1{2q})}
\|\varphi\|_{L^{q'}}^{\gamma}\|\mathcal{M}\varphi\|_{L^{(1-\gamma)q}}^{1-\gamma}.
\end{equation}
Moreover, observe that
\begin{displaymath}
(1-\gamma)q=q' \quad \Leftrightarrow \quad 3q^2-2(7-\beta)q+12=0.
\end{displaymath}
Finally, we conclude from the fact the maximal function is
continuous in $L^{q'}$ that
\begin{equation} \label{Strichartz-theo.8}
\|J_{\beta,\epsilon}\|_{L^q} \lesssim
\epsilon^{-(1+\frac{\beta}2-\frac1{2q})} \|\varphi\|_{L^{q'}},
\end{equation}
which concludes the proof of Theorem \ref{Strichartz-theo}.
\end{proof}

\subsection{The nonlinear Cauchy problem}
The objective of this subsection is to prove the following result.
\begin{theorem} \label{NCP}
Let $s>\frac32$ and $0 < \epsilon \le 1$ be fixed. Then for any
$(w_1^0,w_2^0) \in H^s(\mathbb R^2)\times H^s(\mathbb R^2)$, there
exist a positive time $T=T(\|(w_1^0,w_2^0)\|_{H^s\times H^s})$, a
space $X^s_{T_{\epsilon}}$ such that
\begin{equation} \label{NCP1}
X^s_{T_{\epsilon}} \hookrightarrow C\big([0,T_{\epsilon}] \ ; \ H^s(\mathbb R^2)\times H^s(\mathbb R^2)\big),
\end{equation} and a unique solution $(w_1,w_2)$ to \eqref{KdV type3} in $X^s_{T_{\epsilon}}$ satisfying $(w_1,w_2)_{|_{t=0}}=(w_1^0,w_2^0)$, where $T_{\epsilon} = T\epsilon^{-\frac12}$.

Moreover, for any $T' \in (0,T_{\epsilon})$, there exists a neighborhood $\Omega^s$ of $(w_1^0,w_2^0)$ in $H^s(\mathbb R^2)\times H^s(\mathbb R^2)$ such that the flow map associated to \eqref{KdV type3} is smooth from $\Omega^s$ into $X^s_{T'}$.
\end{theorem}

First, we list some well-known properties of the Riesz transforms.
\begin{proposition} \label{Riesz}
It holds that
\begin{equation} \label{Riesz1}
\|R_jf\|_{L^p} \le C\|f\|_{L^p},
\end{equation}
for all $1<p<\infty$ and $j \in \{1,2\}$, and
\begin{equation} \label{Riesz2}
D^1_x=R_1\partial_{x_1}+R_2\partial_{x_2}.
\end{equation}
\end{proposition}

The following Leibniz' rule for fractional derivative, derived by Kenig, Ponce and Vega in Theorem A.12 of \cite{KPV2}, will also be needed.
\begin{lemma} \label{FLR}
Let $0<\gamma<1$ and $1<p<\infty$. Then
\begin{equation} \label{FLR1}
\big\| D^{\gamma}_x(fg)-fD^{\gamma}_xg-gD^{\gamma}_xf\big\|_{L^p} \lesssim \|g\|_{L^{\infty}} \|D^{\gamma}_xf\|_{L^p},
\end{equation}
for all $f, \ g: \mathbb R^n \rightarrow \mathbb C$.
\end{lemma}

\begin{proof}[Proof of Theorem \ref{NCP}]
We will treat only the most difficult case when $\frac32<s < 2$ and we define $\gamma=s-1 \in (\frac12,1)$. The integral system associated to \eqref{KdV type3} with initial data $(w_1^0,w_2^0)$ can be written as
\begin{equation} \label{NCP2}
\begin{cases}
w_1=\mathcal{F}^+(w_1,w_2):=
U_{\epsilon}^+(t)w_1^0+\epsilon\int_0^t\,U_{\epsilon}^+(t-t')(I+II)(w_1,w_2)(t')dt'\\
w_2=\mathcal{F}^-(w_1,w_2):=
U_{\epsilon}^-(t)w_2^0+\epsilon\int_0^t\,U_{\epsilon}^-(t-t')(-I+II)(w_1,w_2)(t')dt'
\end{cases}
\end{equation}
where the nonlinearities $I$ and $II$ appearing on the right-hand side of \eqref{NCP2} are defined in \eqref{I} and \eqref{II}.

Next, for $0<T\lesssim \epsilon^{-\frac12}$, we fix $\alpha=1-\gamma
\in (0,\frac12)$, $q_{\alpha}$ and $\kappa_{\alpha}$ defined in
Theorem \ref{Strichartz-theo} and consider the following semi-norms
\begin{displaymath} \label{NCP3}
\begin{split}
\lambda_T^1(f)=&\sup_{0\le t\le T} \|f(t)\|_{H^s_x},\\
\lambda_T^2(f)=& \ \epsilon^{\frac37+}\sum_{  |\beta| \le
1}\|\partial_x^{\beta}f\|_{L^{\frac72+}_TL^{\infty}_x}
+\epsilon^{\frac37+}\|D_x^1f\|_{L^{\frac72+}_TL^{\infty}_x}+
\epsilon^{\frac37+}\sum_{l=1,2}\sum_{ |\beta| \le
1}\|\partial_x^{\beta}R_lf\|_{L^{\frac72+}_TL^{\infty}_x}
\\ &+\epsilon^{\frac37+}\sum_{l=1,2}\|D_x^1R_lf\|_{L^{\frac72+}_TL^{\infty}_x},\\
\lambda_T^3(f)=& \ \epsilon^{\kappa_{\alpha}} \sum_{
|\beta|=1}\|D_x^1\partial_x^{\beta}f\|_{L^{q_{\alpha}}_TL^{\infty}_x}
+\epsilon^{\kappa_{\alpha}}\sum_{l=1,2}
\sum_{  |\beta|=1}\|D_x^1\partial_x^{\beta}R_lf\|_{L^{q_{\alpha}}_TL^{\infty}_x}, \\
\lambda_T^4(f)=& \ \epsilon^{\frac12}\sum_{ |\beta| = 1}\sup_{\alpha \in \mathbb Z^2}\|P_{>\epsilon^{-\frac12}}\partial_x^{\beta}D^{1+\gamma}_xf\|_{L^2_{\mathcal{Q}_{\alpha}\times[0,T]}}
\\ & \quad+\epsilon^{\frac12}\sum_{l=1,2}\sum_{ |\beta| = 1}\sup_{\alpha \in \mathbb Z^2}\|P_{>\epsilon^{-\frac12}}\partial_x^{\beta}D^{1+\gamma}_xR_lf\|_{L^2_{\mathcal{Q}_{\alpha}\times[0,T]}},\\
\lambda_T^5(f)=& \ \epsilon^{\frac18}
\Big( \sum_{\alpha \in \mathbb Z^2} \|f\|_{L^{\infty}_{\mathcal{Q}_{\alpha}\times[0,T]}}^2\Big)^{\frac12}
+\epsilon^{\frac18}\sum_{l=1,2}\Big( \sum_{\alpha \in \mathbb Z^2} \|R_lf\|_{L^{\infty}_{\mathcal{Q}_{\alpha}\times[0,T]}}^2\Big)^{\frac12},
\end{split}
\end{displaymath}
where $\big\{Q_{\alpha}\big\}_{\alpha \in \mathbb Z^2}$ denotes a family
of nonoverlapping cubes of unit size such that $\mathbb
R^2=\cup_{\alpha \in \mathbb Z^2}Q_{\alpha}$, as in the precedent subsection.
We also define the Banach space $X^2_T$ by
\begin{displaymath}
X^2_T=\big\{(w_1,w_2) \in C([0,T];H^s(\mathbb R^2)\times H^s(\mathbb R^2))
 \ | \  \|(w_1,w_2)\|_{X^2_T} <\infty \big\}.
\end{displaymath}
where
\begin{displaymath}
\|(w_1,w_2)\|_{X^2_T}=\sum_{j=1}^4\lambda_T^j
(w_1)+\sum_{j=1}^4\lambda_j^T(w_2).
\end{displaymath}
We will show that, for adequate $T$ and $a$, the map $(\mathcal{F}^+,\mathcal{F}^-)$ is a contraction in a closed ball $X_T^2(a)$ in $X_T^2$ of radius $a>0$ and centered at the origin.

Using the integral equation \eqref{NCP2}, Minkowski's integral
inequality, the linear estimates \eqref{locsm.1}, \eqref{maximal.1}
and \eqref{Strichartz-theo.1} with $\alpha=0$ and $\alpha=1-\gamma$,
and the fact that the Riesz transforms are unitary operators in
$H^s$, we deduce that
\begin{equation}\label{NCP4}
\begin{split}
\underset{j=1}{\overset{4}{\sum}}&
\lambda_T^j\big(\mathcal{F}^+(w_1,w_2)\big)+\underset{j=1}{\overset{4}{\sum}}
\lambda_T^j\big(\mathcal{F}^-(w_1,w_2)\big) \\ &
\lesssim \|(w_1^0,w_2^0)\|_{H^s\times H^s}+\epsilon\int\limits_0^T\big(
\|I(w_1,w_2)(t)\|_{H^s_x}+ \|II(w_1,w_2)(t))\|_{H^s_x}\big)dt.
\end{split}
\end{equation}

According to formulas \eqref{I} and \eqref{II}, the nonlinearities $I(w_1,w_2)$ and $II(w_1,w_2)$ contain terms of the three following types: $w_jD_x^1w_k$, $\partial_{x_l}w_jR_lw_k$ and $D^1_x(R_lw_jR_mw_k)$ for $j, \ k, \ l, \ m =1,2.$ For sake of simplicity, we only will present  the computations for a nonlinear term of the first kind, since the other ones can be handled by similar arguments. In other words, we have to estimate the term $\epsilon\int_0^T\|w_jD_x^1w_k\|_{H^s_x}dt$ as a function of the norms defined in \eqref{NCP3}. By the definition of the Sobolev space $H^s(\mathbb R^2)$, we have that
\begin{equation} \label{NCP5}
\epsilon\int_0^T\|w_jD_x^1w_k\|_{H^s_x}dt \lesssim \epsilon\int_0^T\|w_jD_x^1w_k\|_{L^2_x}dt+
\epsilon\int_0^T\|D_x^s\big(w_jD_x^1w_k\big)\|_{L^2_x}dt.
\end{equation}

First, we use H\"older's inequality and the Sobolev embedding
$H^s(\mathbb R^2) \hookrightarrow L^{\infty}(\mathbb R^2)$ to
estimate the first term on the right-hand side of \eqref{NCP5} as
\begin{equation} \label{NCP6}
\begin{split}
\epsilon\int_0^T\|w_jD_x^1w_k\|_{L^2_x}dt &\le \epsilon T
\|w_j\|_{L^{\infty}_TL^{\infty}_x} \|D_x^1w_k\|_{L^{\infty}_TL^2_x}\\
& \le \epsilon T\lambda_T^1(w_j)\lambda_T^1(w_k).
\end{split}
\end{equation}

To treat the second term appearing on the right-hand side of \eqref{NCP5}, we observe from \eqref{Riesz2} that $D_x^s=D_x^{\gamma}R_1\partial_{x_1}+D_x^{\gamma}R_2\partial_{x_2}$. Therefore, it follows from \eqref{Riesz1} that
\begin{equation} \label{NCP7}
\begin{split}
\epsilon&\int_0^T\|D_x^s\big(w_jD_x^1w_k\big)\|_{L^2_x}dt\\ &\le \epsilon\sum_{l=1,2}
\int_0^T\|D_x^{\gamma}\big(\partial_{x_l}w_jD_x^1w_k\big)\|_{L^2_x}dt
+\epsilon\sum_{l=1,2}\int_0^T
\|D_x^{\gamma}\big(w_j\partial_{x_l}D_x^1w_k\big)\|_{L^2_x}dt.
\end{split}
\end{equation}
Then, we deduce from the fractional Leibniz' rule \eqref{FLR1} that
\begin{equation} \label{NCP8}
\begin{split}
\epsilon\int_0^T\|D_x^{\gamma}\big(\partial_{x_l}w_jD_x^1w_k\big)\|_{L^2_x}dt
&\le \epsilon^{\frac47-}T^{\frac57+}
\epsilon^{\frac37+}\|\partial_{x_l}w_j\|_{L^{\frac72+}_TL^{\infty}_x}
\|D_x^{1+\gamma}w_k\|_{L^{\infty}_TL^2_x} \\ &\quad
+\epsilon^{\frac47-}T^{\frac57+}
\epsilon^{\frac37+}\|D_x^1w_k\|_{L^{\frac72+}_TL^{\infty}_x}
\|\partial_{x_l}D_x^{\gamma}w_j\|_{L^{\infty}_TL^2_x}\\ & \le
\epsilon^{\frac47-}T^{\frac57+}\big(\lambda_T^2(w_j)\lambda_T^1(w_k)
+\lambda_T^2(w_k)\lambda_T^1(w_j)\big).
\end{split}
\end{equation}

On the other hand, by using again Leibniz' rule and H\"older's inequality, we observe that
\begin{equation} \label{NCP9}
\begin{split}
\epsilon&\int_0^T
\|D_x^{\gamma}\big(w_j\partial_{x_l}D_x^1w_k\big)\|_{L^2_x}dt\\ &
\le \epsilon\int_0^T\|D_x^{\gamma}w_j\|_{L^2_x}\|\partial_{x_l}D_x^1w_k\|_{L^{\infty}_x}dt
+\epsilon\int_0^T
\|w_j\partial_{x_l}D_x^{1+\gamma}w_k\|_{L^2_x}dt.
\end{split}
\end{equation}
We deal with the first term on the right-hand side of \eqref{NCP9}.
We deduce by using H\"older's inequality in time that
\begin{equation} \label{NCP10}
\begin{split}
\epsilon\int_0^T\|D_x^{\gamma}&w_j\|_{L^2_x}\|\partial_{x_l}D_x^1w_k\|_{L^{\infty}_x}dt
\\ & \le \epsilon^{1-\kappa_{\alpha}}T^{\frac1{q_{\alpha}'}}\|D_x^{\gamma}w_j\|_{L^{\infty}_TL^2_x}
\epsilon^{\kappa_{\alpha}}\|\partial_{x_l}D_x^1w_k\|_{L^{q_{\alpha}}_TL^{\infty}_x}\\
& \le
\epsilon^{1-\kappa_{\alpha}}T^{\frac1{q_{\alpha}'}}\lambda_1^T(w_j)\lambda_3^T(w_k),
\end{split}
\end{equation}
where $q_{\alpha}'$ is the conjugate exponent to $q_{\alpha}$.

Finally, to estimate the second term on the right-hand side of \eqref{NCP9}, we observe that
\begin{equation} \label{NCP11}
\begin{split}
\epsilon&\int_0^T
\|w_j\partial_{x_l}D_x^{1+\gamma}w_k\|_{L^2_x}dt
\\&\le \epsilon \int_0^T
\|w_jP_{\le \epsilon^{-\frac12}}\partial_{x_l}D_x^{1+\gamma}w_k\|_{L^2_x}dt
+\int_0^T
\|w_jP_{> \epsilon^{-\frac12}}\partial_{x_l}D_x^{1+\gamma}w_k\|_{L^2_x}dt.
\end{split}
\end{equation}
On the one hand, we use the Sobolev embedding $H^s(\mathbb R^2)
\hookrightarrow L^{\infty}(\mathbb R^2)$, since $s>1$, to obtain
that
\begin{equation} \label{NCP12}
\begin{split}
\epsilon\int_0^T
\|w_jP_{\le \epsilon^{-\frac12}}\partial_{x_l}D_x^{1+\gamma}w_k\|_{L^2_x}dt
&\le \epsilon T
\|w_j\|_{L^{\infty}_{T,x}}\|P_{\le \epsilon^{-\frac12}}\partial_{x_l}D_x^{1+\gamma}w_k\|_{L^{\infty}_TL^2_x}
\\ & \le \epsilon T
\|w_j\|_{L^{\infty}_{T}H^s_x}\epsilon^{-\frac12}\|D_x^{1+\gamma}w_k\|_{L^{\infty}_TL^2_x}
\\ & \le
\epsilon^{\frac12}T\lambda_T^1(w_j)\lambda_T^1(w_k).
\end{split}
\end{equation}
On the other hand, we deduce from H\"older's inequality that
\begin{equation} \label{NCP13}
\begin{split}
\epsilon\int_0^T
\|w_jP_{> \epsilon^{-\frac12}}&\partial_{x_l}D_x^{1+\gamma}w_k\|_{L^2_x}dt
\\&\le \epsilon T^{\frac12}\Big( \sum_{\alpha \in \mathbb Z^2}\int_{\mathcal{Q}_{\alpha}\times[0,T]}|w_j
P_{> \epsilon^{-\frac12}}\partial_{x_l}D_x^{1+\gamma}w_k|^2dxdt\Big)^{\frac12}\\
&\le \epsilon T^{\frac12}\Big( \sum_{\alpha \in \mathbb Z^2} \|w_j\|_{L^{\infty}_{\mathcal{Q}_{\alpha}\times[0,T]}}^2\Big)^{\frac12}
\sup_{\alpha \in \mathbb Z^2}\|P_{> \epsilon^{-\frac12}}\partial_{x_l}D_x^{1+\gamma}w_k\|_{L^2_{\mathcal{Q}_{\alpha}\times[0,T]}}\\
& \le \epsilon^{\frac38}T^{\frac12}\lambda_T^5(w_j)\lambda_T^4(w_k).
\end{split}
\end{equation}

Thus, it follows gathering \eqref{NCP4}--\eqref{NCP11} that
\begin{equation} \label{NCP12a}
\begin{split}
\big\|\big(\mathcal{F}^+(w_1,w_2),&\mathcal{F}^-(w_1,w_2)\big)\big\|_{X^2_T}
\le c\|(w_1^0,w_2^0)\|_{H^2\times
H^2}\\&+c\big(\epsilon^{\frac47-}T^{\frac57+}+
\epsilon^{1-\kappa_{\alpha}}T^{\frac1{q_{\alpha}'}}+\epsilon^{\frac12}T+
\epsilon^{\frac38}T^{\frac12}\big)\|(w_1,w_2)\|_{X^2_T}^2,
\end{split}
\end{equation}
and similarly,
\begin{equation} \label{NCP13a}
\begin{split}
\big\|\big(\mathcal{F}^+&(w_1,w_2),\mathcal{F}^-(w_1,w_2)\big)-
\big(\mathcal{F}^+(\tilde{w}_1,\tilde{w}_2),\mathcal{F}^-
(\tilde{w}_1,\tilde{w}_2)\big)\big\|_{X^2_T} \\
&\le c\big(\epsilon^{\frac47-}T^{\frac57+}+
\epsilon^{1-\kappa_{\alpha}}T^{\frac1{q_{\alpha}'}}+\epsilon^{\frac12}T+
\epsilon^{\frac38}T^{\frac12}\big)\\ & \quad \times
\big(\|(w_1,w_2)\|_{X^2_T}+\|(\tilde{w}_1,\tilde{w}_2)\|_{X^2_T}
\big)\|(w_1,w_2)-(\tilde{w}_1,\tilde{w}_2)\|_{X^2_T}.
\end{split}
\end{equation}
Observe that
$$c\big(\epsilon^{\frac47-}T^{\frac57+}+
\epsilon^{1-\kappa_{\alpha}}T^{\frac1{q_{\alpha}'}}+\epsilon^{\frac12}T+
\epsilon^{\frac38}T^{\frac12}\big) \lesssim 1,$$
since $T \lesssim \epsilon^{-\frac12}$. Therefore, we
deduce by choosing
\begin{displaymath}
a=2c\|(w_1^0,w_2^0)\|_{H^2\times H^2} \quad \text{and} \quad
T \sim \epsilon^{-\frac12}
\end{displaymath}
that $(\mathcal{F}^+,\mathcal{F}^-)$ defines a contraction in the space $X^2_T(a)$, which concludes the proof of
Proposition \ref{NCP} by the Picard fixed point theorem.
\end{proof}

\section{Other strongly dispersive Boussinesq systems}
Other two-dimensional Boussinesq systems  can be classified
according to the order of the eigenvalues of the dispersion matrix
(see Introduction). As we already mentioned,  when $b,d> 0$, the
well-posedness on time scales of order $1/{\sqrt{\epsilon}}$ has
been established  in \cite{DMS1} for initial data in low order
Sobolev spaces and on time interval of order $1/{\epsilon}$ in
\cite{MSZ} for smooth initial data (with loss of derivatives). We
thus restrict ourselves to the case where at least one of the
coefficients $b$ or $d$ vanishes, emphasizing the strongest
dispersive case corresponding to eigenvalues of order $2$ (the
so-called \lq\lq Schr\"{o}dinger\rq\rq \ type systems) where $a<0$
and $c<0.$

The local well posedness for some of those systems has been established in the two-dimensional case by Cung The Anh \cite{CTA},
but without looking for the dependence of the existence time with respect to $\epsilon$. Note however that the proof given in
\cite{CTA} for the case where $d=0$ (Theorem 3.5) does not seem to be correct.

We will complete the analysis, in particular by  checking the $\epsilon$  dependence. It turns out that when $d>0,$ one can
establish the local well-posedness of  Boussinesq \lq\lq Schr\"{o}dinger\rq\rq \ type systems on time intervals of order
$1/{\sqrt{\epsilon}}$ by elementary energy methods, in relatively low order Sobolev spaces.

We first state a useful lemma.

\begin{lemma}\label{impa}
Let $s\geq3/2$. Then there exists $C>0$ such that for any ${\bf v}\in H^{s+1}(\R^2)$ and $\eta \in H^s(\R^2),$
\begin{equation} \label{impa.5}
|(\Lambda ^s\nabla\cdot (\eta {\bf v}),\Lambda ^s \eta)|\leq C\|{\bf v}\|_{H^ {s+1}}\|\eta\|_{H^s}^2.
 \end{equation}
\end {lemma}

The proof of this lemma is based on the following commutator estimate obtained  by Ponce in Lemma 2.3
of \cite{Po}.
\begin{lemma}\label{Ponce}
If $s>1$ and $1<p<\infty$, then
\begin{equation} \label{Ponce1}
\|[\Lambda^s;f]g\|_{L^p} \leq C\big(\|\nabla f\|_{L^{p_1}}\|\Lambda^{s-1}g\|_{L^{p_2}}+
\|\Lambda^s f\|_{L^{p_3}}\|g\|_{L^{p_4}} \big),
\end{equation}
where $p_1, \ p_4 \in (1,+\infty]$ such that
$$\frac1{p_1}+\frac1{p_2}=\frac1{p_3}+\frac1{p_4}=\frac1p.$$
\end{lemma}

\begin{proof}[Proof of Lemma \ref{impa}]
Fix $s \ge \frac32$. First observe that
\begin{equation} \label{impa.1}
(\Lambda ^s\nabla\cdot (\eta {\bf v}),\Lambda ^s \eta)=(\Lambda ^s\partial_{x_1}(\eta v_1),\Lambda ^s \eta)
+(\Lambda ^s\partial_{x_2}(\eta v_2),\Lambda ^s \eta).
\end{equation}
Moreover, we can rewrite the terms appearing on the right-hand side of \eqref{impa.1} as
\begin{equation} \label{impa.2}
(\Lambda ^s\partial_{x_j}(\eta v_j),\Lambda ^s \eta)\!=
(\Lambda ^s(\eta\partial_{x_j} v_j),\Lambda ^s \eta)
\!+([\Lambda^s,v_j]\partial_{x_j}\eta,\Lambda ^s \eta)+(v_j\Lambda ^s\partial_{x_j}\eta,\Lambda ^s \eta).
\end{equation}
We use Lemma \ref{FLR} and the Sobolev embedding
$H^{\alpha}(\mathbb R^2) \hookrightarrow L^{\infty}(\mathbb R^2)$, for $\alpha>1$,
to estimate the first term on the right-hand side of \eqref{impa.2} as
\begin{equation} \label{impa.2a}
\big|(\Lambda ^s(\eta\partial_{x_j} v_j),\Lambda ^s \eta)\big|
\leq C \|{\bf v}\|_{H^ {s+1}}\|\eta\|_{H^s}^2.
\end{equation}
On the other hand, it follows from Lemma \ref{Ponce} with $p_1=\infty$, $p_2=2$, $p_3=p_4=4$, and
the Sobolev embedding
$H^{\frac12}(\mathbb R^2) \hookrightarrow L^4(\mathbb R^2)$ that
\begin{equation} \label{impa.3}
\begin{split}
\big|([\Lambda^s,v_j]\partial_{x_j}&\eta,\Lambda ^s \eta)\big|\\
&\leq C\big(\|\nabla v_j\|_{L^{\infty}}\|\Lambda^{s-1}\partial_{x_j}\eta\|_{L^{2}}+
\|\Lambda^s v_j\|_{L^4}\|\partial_{x_j}\eta\|_{L^4} \big)\|\Lambda^s\eta\|_{L^{2}} \\
& \leq C \|{\bf v}\|_{H^ {s+1}}\|\eta\|_{H^s}^2.
\end{split}
\end{equation}
Finally, we have integrating by parts that
\begin{displaymath}
(v_j\Lambda ^s\partial_{x_j}\eta,,\Lambda ^s \eta)
=-\frac12(\partial_{x_j}v_j\Lambda ^s\eta,\Lambda ^s \eta),
\end{displaymath}
which implies by H\"older's inequality and the Sobolev embedding
$H^{\alpha}(\mathbb R^2) \hookrightarrow L^{\infty}(\mathbb R^2)$, for $\alpha>1$, that
\begin{equation}  \label{impa.4}
\big|(v_j\Lambda ^s\partial_{x_j}\eta,\Lambda ^s \eta)\big|\leq C \|{\bf v}\|_{H^ {s+1}}\|\eta\|_{H^s}^2.
\end{equation}
Therefore, estimate \eqref{impa.5} is deduced gathering \eqref{impa.1}--\eqref{impa.4}.
\end{proof}


%

\subsection{Schr\"{o}dinger type Boussinesq systems when $d>0$}

As already mentioned, attention is focused here in systems where the eigenvalues $\lambda_{\pm}(\xi)$ are of order $2$. Because of
the constraint on $a,b,c,d$ the only possibilities are  $a<0,$ $c<0,$ $b=0$, $d>0$ or $a<0,$ $c<0,$ $b>0,$ $d=0$.
\footnote{When surface tension is large enough so that $\tau>1/3$, we have the admissible system $a=b=d=0,\; c=\tau -1/3$ but
we will note consider it here.}
We will study here the first case and thus consider  the systems

\begin{equation}\label{BO}
\left\lbrace
\begin{array}{l}
    \eta_t+\nabla \cdot {\bf v}+\epsilon \lbrack\nabla\cdot(\eta {\bf v})+a \nabla\cdot \Delta{\bf v}\rbrack=0 \\
    {\bf v}_t+\nabla \eta+\epsilon\lbrack \frac{1}{2}\nabla |{\bf v}|^2+ c\nabla\Delta \eta- d\Delta {\bf v}_t\rbrack=0,
\end{array}\right.
\end{equation}
where $a<0,$ $c<0,$ $d>0.$

Those systems are called in \cite{BCS2} ``Benjamin-Ono type systems" since  the dispersion matrix has, in the one-dimensional
case and when $\epsilon =1$,  eigenvalues  equal to $\pm ((\frac{ac}{d})^{\frac{1}{2}}k|k|+r(k))$ where $r(k)$ is
$0(\frac{1}{|k|})$ as $k \to \infty.$

In the two-dimensional case, the dispersion matrix in Fourier variables reads

 \begin{displaymath}
\widehat{A}(\xi_1,\xi_2)=i\begin{pmatrix} 0 & \xi_1(1-\epsilon a |\xi|^2) & \xi_2(1-\epsilon a |\xi|^2) \\
                 \frac{\xi_1(1-\epsilon c|\xi|^2)}{1+\epsilon d|\xi|^2}  & 0 & 0 \\
                 \\
                 \frac{\xi_2(1-\epsilon c|\xi|^2)}{1+\epsilon d|\xi|^2}    & 0 & 0 \end{pmatrix}.
\end{displaymath}

\vspace{0.3cm}

The corresponding nonzero eigenvalues are

$$\lambda_{\pm}=\pm  |\xi|\left(\frac{(1-\epsilon a|\xi|^2)(1-\epsilon c|\xi|^2)}{1+\epsilon d|\xi|^2}\right)^{\frac{1}{2}}.$$

For $\epsilon$ fixed, $\lambda_{\pm}$  are equivalent as $|\xi|\to\infty$ to $\pm \epsilon\left(\frac{ac}{d}\right )^{\frac{1}{2}}|\xi|^2,$
which explains the terminology \lq\lq Schr\"{o}dinger type" Boussinesq systems.

After diagonalization of the linear part, \eqref{BO} can thus be written as a (nonlinearly coupled) system of two Schr\"{o}dinger type
nonlocal equations involving derivatives of the unknowns, and we might think of solving the Cauchy problem by applying the general
results of \cite{KPV4}. We will   refrain to do that here because  an elementary energy method on the original formulation \eqref{BO},
which is  the extension to the two-dimensional case of the corresponding one-dimensional result in \cite{BCS2} applies when $d>0,$ and
when $d=0$ for curlfree velocities.

\begin{theorem}\label{couveflor}  Let  $s\geq 3/2$ and $0 < \epsilon\le1$ be given.

(i) Assume that $a=c$.
Then, for every $(\eta_0,{\bf v}_0)\in H^s(\R^2)\times H^{s+1}(\R^2)^2,$ there exist $T=T(||\eta_0||_{H^s}, ||{\bf v}_0||_{H^{s+1}})$
and a unique solution $$(\eta,{\bf v})\in C([0,T_{\epsilon}]; H^s(\R^2))\times  C([0,T_{\epsilon}]; H^{s+1}(\R^2)^2),$$ where
$T_{\epsilon}=T\epsilon^{-1/2}$, of \eqref{BO} with initial data $(\eta_0,{\bf v}_0),$ which is uniformly
bounded on $\lbrack 0,T_{\epsilon}\rbrack.$

(ii) Assume that  $a\neq c$.
Then, for every $(\eta_0,{\bf v}_0)\in H^{s+1}(\R^2)\times H^{s+2}(\R^2)^2,$ there exist
$T=T(||\eta_0||_{H^{s+1}}, ||{\bf v}_0||_{H^{s+2}})$ and a unique solution
$$(\eta,{\bf v})\in C([0,T_{\epsilon}]; H^{s+1}(\R^2)) \times  C([0,T_{\epsilon}]; H^{s+2}(\R^2)^2), $$
where $T_{\epsilon}=T\epsilon^{-1/2}$ of \eqref{BO} with initial data $(\eta_0,{\bf v}_0)$  which is uniformly
bounded on $\lbrack 0,T_{\epsilon}\rbrack.$
\end{theorem}

\begin{proof}
\textit{(i) Uniqueness in the class $C([0,T]; H^{3/2}(\R^2))\times  C([0,T]; H^{5/2}(\R^2))^2$.}

In what follows we will use freely the embedding $H^{\frac{1}{2}}(\R^2)\hookrightarrow L^4(\R^2)$ and
thus $H^{\frac{3}{2}}(\R^2)\hookrightarrow W^{1,4}(\R^2)$ (see \cite{BL}).

Plainly it suffices to consider $\epsilon =1$. Let $(\eta_1,{\bf v}_1),$  $(\eta_2,{\bf v}_2)$ two solutions in the
class $C([0,T]; H^{3/2}(\R^2))\times  C([0,T]; H^{5/2}(\R^2))^2$ and let $N=\eta_1-\eta_2,$
${\bf V}= {\bf v}_1-{\bf v}_2.$ We have therefore

\begin{equation}\label{BObis}
\left\lbrace
\begin{array}{l}
    N_t+\nabla \cdot {\bf V}+\epsilon \lbrack\nabla\cdot(N {\bf v}_1+\eta_2 {\bf V})+a \nabla\cdot \Delta{\bf V}\rbrack=0 \\
    {\bf V}_t+\nabla N+\epsilon\lbrack \frac{1}{2}\nabla( {\bf V}\cdot ( {\bf v}_1+{\bf v}_2))+ c\nabla\Delta N- d\Delta {\bf V}_t\rbrack=0,
\end{array}\right.
\end{equation}

We take the $L^2$ scalar product of the first equation by $|c|N$, of the second one by $|a|{\bf V}$ and add the resulting equations.
The contributions of the cubic dispersive terms cancel out \footnote{One term has to be interpreted as a $H^{1/2}-H^{-1/2}$ duality.}
and we find

\begin{equation}\label{Gro}
\begin{split}
   &\frac{1}{2}\frac{d}{dt}\int_{\R^2}(|c||N|^2+|a||{\bf V}|^2+d|\nabla {\bf V}|^2)=\int_{\R^2}(aN\nabla \cdot{{\bf V}}+c\nabla N\cdot {\bf V})\\
   &-\int_{\R^2}\nabla \cdot{(N{\bf v}_1+\eta_2{\bf V}})N-\frac{1}{2}\int_{\R^2}\nabla( {\bf V}\cdot ( {\bf v}_1+{\bf v}_2))\cdot{\bf V}.
   \end{split}
   \end{equation}

   One has successively

   \begin{equation}
\left| \int_{\R^2}(|a|N\nabla \cdot{{\bf V}}+|c|\nabla N\cdot {\bf V})\right|=\left|
\int_{\R^2}(|a|-|c|)N\nabla \cdot{{\bf V}}\right|\leq C\|N\|_{L^2}\|\nabla {\bf V}\|_{L^2}.
   \end{equation}

   Then
 $$  \int_{\R^2}\nabla \cdot{(N{\bf v}_1+\eta_2{\bf V}})N= \int_{\R^2}(N^2\nabla \cdot{{\bf v}_1}+\eta_2N\nabla \cdot{{\bf V}}+(\nabla N\cdot{\bf v}_1)N+(\nabla \eta_2\cdot{\bf V})N.$$

Observing that  $\int_{\R^2}(\nabla N\cdot{\bf v}_1)N=-\frac{1}{2}\int_{\R^2}N^2 \nabla \cdot{{\bf v}_1},$ the contribution of this term
is upper bounded by

\begin{displaymath}
\begin{split}
C(\|\nabla \cdot{{{\bf v}_1}}\|_{L^{\infty}}\|N\|^2_{L^2}+\|\eta_2\|_{L^{\infty}}\|N\|_{L^2}\|\nabla {\bf V}\|_{L^2}&+\|{\bf V}\|_{L^4}\|\nabla \eta_2\|_{L^4}\|N\|_{L^2}) \\
&\leq C(\|N\|_{L^2}^2+\|\nabla {\bf V}\|_{L^2}^2).
\end{split}
\end{displaymath}

Finally,

$$
\left| \int_{\R^2}\nabla( {\bf V}\cdot ( {\bf v}_1+{\bf v}_2))\cdot{\bf V}\right|
\leq C(\|{\bf v}_1\|_{W^{1,\infty}}+\|{\bf v}_2\|_{W^{1,\infty}})\|{\bf V}\|_{H^1}^2.$$

The result follows from \eqref{Gro} and Gronwall's lemma.\\

\noindent \textit{(ii) Existence.} 
We will  focus only on the derivation of the energy estimates checking the dependence of the existence interval on $\epsilon$. The
complete proof would result from a standard compactness argument implemented on a regularized version of the system (for instance by
truncating the high frequencies in the differential operators). The strong continuity in time and the continuity of the flow would
result from a standard application of the Bona-Smith trick. \\


\noindent \textit{Energy estimates in the case $a=c$.}
We apply the operator $\Lambda^s$ to both equations of  \eqref {BO}, multiply the first one by $|a|\Lambda^s \eta$, take the  scalar
product of the second one by $|a|\Lambda^s{\bf V}$, integrate and sum. All the linear  terms cancel out and it remains to estimate

$$I_1=\epsilon\int_{\R^2}( \Lambda^s\nabla \cdot({\eta{\bf v}})\Lambda^s \eta+\frac{1}{2}\Lambda^s\nabla|{\bf v}|^2\cdot \Lambda^s{\bf v})=\epsilon(I_2+I_3).$$

Lemma \ref{impa} implies that
$$|I_2|\leq  C\|{\bf v}\|_{H^ {s+1}}\|\eta\|_{H^s}^2.$$
Finally, by the Kato-Ponce commutator lemma (\cite{kp}),
$$|I_3| \leq C\|{\bf v}\|_{H^{s+1}}\|{\bf v}\|^2_{H^s}.$$
Combining those estimates leads to the differential inequality
\begin{equation} \label{couveflor1}
\frac{d}{dt}\big\lbrack \|\eta\|_{H^s}^2+\|{\bf v}\|^2_{H^s}
+\epsilon \|{\bf v}\|^2_{H^{s+1}}\big\rbrack
\leq C\epsilon\big(\|{\bf v}\|_{H^{s+1}}\|\eta\|^2_{H^s}+\|{\bf v}\|_{H^s}^2\|{\bf v}\|_{H^{s+1}}\big).
\end{equation}
If we define $Y(t)= \|\eta\|_{H^s}^2+\|{\bf v}\|^2_{H^s}+\epsilon \|{\bf v}\|^2_{H^{s+1}}$,
then inequality \eqref{couveflor1} can
be rewritten as
$$Y'(t)\leq C\epsilon^{1/2} Y(t) ^{3/2},$$
which in turn implies an a priori bound on a time interval  $\lbrack 0,\frac{T}{\sqrt{\epsilon}}\rbrack$ where
$T=T(||\eta_0||_{H^s}, ||{\bf v}_0||_{H^{s+1}})$.\\

\noindent \textit{Energy estimates in the case $a\neq c$.}
We apply the operator $\Lambda^s$ to both equations of  \eqref {BO}, multiply the first one by $\Lambda^s \eta+c\epsilon\Delta\Lambda^s \eta$,
take the  scalar product of the second one by $\Lambda^s {\bf v}+a\epsilon\Delta\Lambda^s  {\bf v}$, integrate and sum to obtain that
\begin{equation} \label{couveflor2}
\begin{split}
&\frac12 \frac{d}{dt}\big\lbrack \|\eta\|_{H^s}^2-c\epsilon\|\nabla \eta\|_{H^s}^2+\|{\bf v}\|^2_{H^s}+(d-a)\epsilon\|{\nabla\bf v}\|^2_{H^s}-ad\epsilon^2 \|\Delta{\bf v}\|^2_{H^{s}}\big\rbrack \\
&=  \int \nabla\cdot \Lambda^s{\bf v} \Lambda^s\eta+c\epsilon\int \nabla\cdot \Lambda^s{\bf v} \Delta\Lambda^s\eta+
 \int \nabla\Lambda^s\eta\cdot \Lambda^s{\bf v} \\
&\quad+a\epsilon \int \nabla\Lambda^s\eta\cdot \Delta\Lambda^s{\bf v}
+a\epsilon\int\nabla\cdot\Delta\Lambda^s{\bf v}\Lambda^s\eta+ac\epsilon^2\int\nabla\cdot\Delta\Lambda^s{\bf v}\Delta\Lambda^s\eta\\
&\quad+c\epsilon\int\nabla\Delta\Lambda^s\eta\cdot\Lambda^s{\bf v}
+ac\epsilon^2\int\nabla\Delta\Lambda^s\eta\cdot\Delta\Lambda^s{\bf v}+\epsilon\int \Lambda^s\nabla\cdot(\eta{\bf v})\Lambda^s\eta\\
&\quad+\epsilon^2 c\int \Lambda^s\nabla\cdot(\eta{\bf v})\Delta\Lambda^s\eta
+\frac{\epsilon} {2}\int\Lambda^s\nabla(|{\bf v}|^2)\cdot\Lambda^s{\bf v}+\frac{\epsilon^2 a}2\int\Lambda^s\nabla(|{\bf v}|^2)\cdot\Delta\Lambda^s{\bf v}\\
&:= \sum_{l=1}^{12}J_l.
\end{split}
\end{equation}
Now, observe integrating by parts that $J_1+J_3=J_2+J_7=J_4+J_5=J_6+J_8=0.$

Moreover, we deduce from Lemma \ref{impa} that
\begin{equation} \label{couveflor3}
|J_9| \le \epsilon \|{\bf v}\|_{H^ {s+1}}\|\eta\|_{H^s}^2,
\end{equation}
and
\begin{equation} \label{couveflor4}
|J_{10}| \le \epsilon^2\|{\bf v}\|_{H^ {s+2}}\|\eta\|_{H^{s+1}}^2.
\end{equation}
On the other hand, it follows integrating by parts and using Lemma \ref{Ponce} that
\begin{equation} \label{couveflor5}
|J_{11}| \le \epsilon \|{\bf v}\|_{H^ {s+1}}\|{\bf v}\|_{H^{s}}^2,
\end{equation}
and
\begin{equation} \label{couveflor6}
|J_{12}| \le \epsilon^2\|{\bf v}\|_{H^ {s+2}}\|{\bf v}\|_{H^{s+1}}^2.
\end{equation}

Finally, let define $Y$ by $$Y(t)=\|\eta\|_{H^s}^2-c\|\nabla \eta\|_{H^s}^2+\|{\bf v}\|^2_{H^s}+(d-a)\epsilon\|{\nabla\bf v}\|^2_{H^s}-ad\epsilon^2 \|\Delta{\bf v}\|^2_{H^{s}}.$$ Then it follows gathering \eqref{couveflor2}--\eqref{couveflor6} that
$$Y'(t)\leq C\epsilon^{1/2} Y(t) ^{3/2},$$
which in turn implies an a uniform in $\epsilon$ priori bound on a time interval $\lbrack 0,\frac{T}{\sqrt{\epsilon}}\rbrack.$
\end{proof}



\subsection{Schr\"{o}dinger type Boussinesq systems when $d=0$ with curl free velocity field.}

We thus consider the case $a<0,$ $c<0,$ $b>0,$ $d=0.$


\begin{equation}
\label{d=0} \left\lbrace
\begin{array}{l}
\eta_t+\nabla \cdot {\bf v}+\epsilon \lbrack\nabla\cdot(\eta {\bf v})+a \nabla\cdot \Delta{\bf v}-b\Delta \eta_t\rbrack=0 \\
{\bf v}_t+\nabla \eta+\epsilon\lbrack \frac{1}{2}\nabla |{\bf
v}|^2+c\nabla \Delta \eta\rbrack=0.
\end{array}\right.
\end{equation}

We will here restrict the velocity {\bf v} to irrotational motions,
a situation  which, as we already indicated is relevant since the variable ${\bf v}$ in the
Boussinesq systems is curlfree up to a $O(\epsilon^2)$ term. Note also that the curlfree condition
is preserved by the evolution of \eqref{d=0}. When ${\bf v}$ is
curlfree, the term $\frac{1}{2}\nabla |{\bf v}|^2$ in the second
equation of \eqref{d=0} writes  as two transport equations namely $({\bf v}\cdot\nabla v_1, {\bf
v}\cdot\nabla v_2)^T$ where ${\bf v}=(v_1,v_2)^T.$ This will permit
to perform energy estimates on ${\bf v}.$

\begin{theorem}\label{sindicato}  Let  $s>2$ and $0 < \epsilon \le 1$ be given.

(i) Assume that $a=c$. Then, for every $(\eta_0,{\bf v}_0)\in
H^{s+1}(\R^2)\times H^s(\R^2)^2$ with $\curl {\bf v}_0=0$ there
exist $T=T(||\eta_0||_{H^{s+1}}, ||{\bf v}_0||_{H^s})$  and a unique
solution $$(\eta,{\bf v})\in C([0,T_{\epsilon}];
H^{s+1}(\R^2))\times  C([0,T_{\epsilon}]; H^s(\R^2)^2),$$ where
$T_{\epsilon}=T\epsilon^{-1/2}$, of \eqref{BO} with initial data
$(\eta_0,{\bf v}_0),$ which is uniformly bounded on $\lbrack
0,T_{\epsilon}\rbrack.$

(ii) Assume that  $a\neq c$. Then, for every $(\eta_0,{\bf v}_0)\in
H^{s+2}(\R^2)\times H^{s+1}(\R^2)^2$ with $\curl {\bf v}_0=0$  there
exist  $T=T(||\eta_0||_{H^{s+2}}, ||{\bf v}_0||_{H^{s+1}})$ and a
unique solution $$(\eta,{\bf v})\in C([0,T_{\epsilon}];
H^{s+2}(\R^2)) \times  C([0,T_{\epsilon}]; H^{s+1}(\R^2)^2), $$
where $T_{\epsilon}=T\epsilon^{-1/2}$ of \eqref{BO} with initial
data $(\eta_0,{\bf v}_0)$  which is uniformly bounded on $\lbrack
0,T_{\epsilon}\rbrack.$
\end{theorem}

The proof of Theorem \ref{sindicato} is very similar to that of Theorem \ref{couveflor} and we only sketch it.
We will use the following estimate.

\begin{lemma}\label{curlfree}

Let $s>2$ and ${\bf v}\in H^s(\R^2)$ such that $\curl {\bf v}=0.$ Then

$$\Big|\int_{\R^2}\Lambda^s \nabla |{\bf v}|^2\cdot \Lambda^s {\bf
v}\Big|\leq C\|{\bf v}\|^3_{H^s}.$$
\end{lemma}

\begin{proof}
Let ${\bf v}=(v_1,v_2)^T.$ Then $\partial_{x_1} v_2=\partial_{x_2}
v_1$ and
 $$\frac{1}{2}\nabla|{\bf v}|^2=(v_1\partial _{x_1} v_1+v_2\partial_{x_2} v_1, v_1\partial_{x_1} v_2+v_2\partial_{x_2} v_2)^T.$$
The estimate then follows immediately from integration by parts, the Kato-Ponce commutator lemma and the embedding
 $H^{s-1}(\R^2)\hookrightarrow L^{\infty}(\R^2).$
\end{proof}

\begin{proof}[Proof of Theorem \ref{sindicato}]
\noindent (i) {\it Uniqueness.} Again we consider $\epsilon =1$. Let $(\eta_1,{\bf v}_1),$  $(\eta_2,{\bf v}_2)$ two solutions  and let $N=\eta_1-\eta_2,$
${\bf V}= {\bf v}_1-{\bf v}_2.$ We have therefore
\begin{equation}\label{BOter}
\left\lbrace
\begin{array}{l}
    N_t+\nabla \cdot {\bf V}+\epsilon \lbrack\nabla\cdot(N {\bf v}_1+\eta_2 {\bf V})+a \nabla\cdot \Delta{\bf V}-b \Delta N_t \rbrack=0 \\
    {\bf V}_t+\nabla N+\epsilon\lbrack \frac{1}{2}\nabla( {\bf V}\cdot ( {\bf v}_1+{\bf v}_2))+c\nabla \Delta N\rbrack=0.
\end{array}\right.
\end{equation}
We multiply the first equation by $-cN$, take scalar product of the second one by $-a{\bf V}$, integrate and sum. One obtains
after integrations by parts
\begin{displaymath}
\begin{split}
\frac{1}{2}\frac{d}{dt}&\int_{\R^2}\lbrack |c|N^2+|a||{\bf V}|^2+|c|b|\nabla N|^2\rbrack \\
&\leq C(\|{\bf v}_1\|_{L^{\infty}},\|\eta_2\|_{L^{\infty}}, \|\nabla {\bf v}_1\|_{L^{\infty}},
\|\nabla{\bf v}_2\|_{L^{\infty}})\lbrack \|N\|_{L^2}^2+\|\nabla N\|_{L^2}^2+\|{\bf V}\|_{L^2}^2\rbrack,
\end{split}
\end{displaymath}
and the result follows from Gronwall's lemma.

\vspace{0.3cm}
\noindent (ii) {\it Existence.} We just indicate how to obtain the energy estimates in the case $a=c$. We apply the operator $\Lambda^s$ to both equations
of  \eqref {BO}, multiply the first one by $\Lambda^s \eta$, take the  scalar product of the second one by $\Lambda^s{\bf v}$,
integrate and sum to get using Kato-Ponce Lemma and Lemma \ref{curlfree}

\begin{equation}\label{rio}
\begin{split}
\frac{d}{dt}\lbrack  \|\Lambda ^s \eta\|_{L^2}^2&+ \|\Lambda ^s{\bf
v}\|_{L^2}^2+\epsilon b \|\nabla \Lambda^s \eta\|_{L^2}^2\rbrack
 \\ &\leq C \epsilon\Big|\int_{\R^2} \Lambda^s \nabla\cdot(\eta {\bf v})\Lambda ^s \eta +\Lambda^s \nabla|{\bf v}|^2\cdot
 \Lambda ^s {\bf v}\Big|\\
 &\leq C \epsilon\big( \|\nabla {\bf v}\|_{L^{\infty}}\|\Lambda ^s \eta\|_{L^2}^2+ \|{\bf v}\|^3_{H^s}\big),
 \end{split}
  \end{equation}
from which one obtain the expected a priori uniform estimate on a
time interval $\lbrack 0, T/\sqrt{\epsilon}\rbrack.$

When $a\neq c$ one has to modify the proof as in Theorem
\ref{couveflor}.
\end{proof}

\section{Final comments}

To keep this paper short and to maintain a certain unity, we have refrain here to consider the other Boussinesq systems,
which have eigenvalues $\lambda_{\pm}$ of order $1$ or less. For instance the cases $a=c=b=0,$ $d=1/3$ or $a=c=d=0,$ $b=1/3$
are version of Boussinesq original system. The former case has been studied in the one-dimensional case by Schonbek \cite{Sc}
and Amick \cite{A} who obtained global well-posedness by viewing the system  as a (dispersive) perturbation of the Saint-Venant
(shallow-water) hyperbolic system.

In both cases, no results seem to be available in  two-dimensions.

\begin{merci}
The Authors were partially  supported by the Brazilian-French program in mathematics. J.-C. S. acknowledges support from
the project ANR-07-BLAN-0250 of the Agence Nationale de la Recherche. F. L. was partially supported by CNPq and FAPERJ/Brazil.
\end{merci}

\end{document}